\documentclass[a4paper,12pt]{amsart}

\usepackage[latin1]{inputenc}
\usepackage[T1]{fontenc}
\usepackage[english]{babel}
\usepackage{amsmath,amssymb,amsthm,mathtools}

\usepackage{latexsym}
\usepackage{delarray}
\usepackage{bbm}
\usepackage{scrextend}
\usepackage{hyperref}
\usepackage[pdftex,usenames,dvipsnames]{xcolor}
\usepackage{datetime}
\usepackage{mathrsfs}
\usepackage{enumitem}
\usepackage{tikz}

\usepackage{MnSymbol,wasysym}

\newtheorem{Th}{Theorem}[section]
\newtheorem{Prop}[Th]{Proposition}
\newtheorem{Lem}[Th]{Lemma}

\newtheorem{Rem}[Th]{Remark}

\newtheorem*{DefBL}{Banach lattice}
\newtheorem*{DefKFS}{K\"othe function space}
\newtheorem*{DefHLP}{Hardy-Littlewood property}
\newtheorem*{DefFP}{Fatou property}

\newtheorem*{Rem*}{Remark}

\newcommand{\Rn}{\mathbb{R}^n}
\newcommand{\NN}{\mathbb{N}}

\newcommand{\CA}{\mathcal{A}}

\allowdisplaybreaks 

\title
[The Hardy-Littlewood property and maximal operators]
{The Hardy-Littlewood property and maximal operators associated with the inverse Gauss measure }

\author[J. J. Betancor]{J. J. Betancor}
\address{Jorge J. Betancor\newline
	Departamento de An\'alisis Matem\'atico, Universidad de La Laguna,\newline
	Campus de Anchieta, Avda. Astrof\'isico S\'anchez, s/n,\newline
	38721 La Laguna (Sta. Cruz de Tenerife), Spain}
\email{jbetanco@ull.es}

\author[A. J. Castro]{A. J. Castro}
\address{\newline
       Alejandro J. Castro \newline
       Department of  Mathematics,
       Nazarbayev University, \newline
		010000 Nur-Sultan, Kazakhstan}
\email{alejandro.castilla@nu.edu.kz}

\author[M. De Le\'on-Contreras]{M. De Le\'on-Contreras}
\address{\newline
       Marta De Le\'on-Contreras \newline
       Department of  Mathematics and Statistics,
       University of Reading, \newline
	   Reading RG6 6AX, United Kingdom}
\email{m.deleoncontreras@reading.ac.uk }

\keywords{Hardy-Littlewood property, Maximal operators, inverse Gauss measure}

\subjclass[2010]
{42B25, 46B42}

 \thanks{
J. J. B. was partially supported by PID2019-106093GB-I00,
A. J. C.  by the Nazarbayev University FDCRGP  110119FD4544 and 
M. D L-C by EPSRC Research Grant EP/S029486/1.}

\begin{document}

\maketitle

\begin{abstract}
In this paper we characterize the Banach lattices with the Hardy-Littlewood property by using maximal operators defined by semigroups of operators associated with the inverse Gauss measure.
\end{abstract}

\setcounter{secnumdepth}{3}
\setcounter{tocdepth}{3}


\section{Introduction}

We consider the Euclidean space $\Rn$ endowed with the measure whose density $\gamma_{-1}$ with respect to the Lebesgue measure is $$
\gamma_{-1}(x)
:=\pi^{n/2} e^{|x|^2},
\quad x \in \Rn.
$$
This measure is called \textit{inverse Gauss measure}.
We recall that the \textit{Gauss measure} is the one defined by the density function
$$
\gamma_{1}(x)
:=\pi^{n/2} e^{-|x|^2},
\quad x \in \Rn,
$$
with respect to the Lebesgue measure. Note that the Gauss measure is a probability measure on $\Rn$, while the inverse Gauss measure is not even finite in $\Rn$. Furthermore, if for every $x \in \Rn$ and $r>0$, $B(x,r)$ denotes the ball in $\Rn$ with center at $x$ and radius $r$, the measure $\gamma_{-1}(B(x,r))$ of $B(x,r)$ grows more than exponentially with $r$, as $r$ tends to $\infty$. These facts make it interesting to study harmonic analysis in the inverse Gaussian setting $(\Rn, \gamma_{-1})$. Some ideas and methods that are needed in $(\Rn, \gamma_{-1})$ can be useful in the study of harmonic analysis on manifolds where the volume growth is superexponential.\\

The \textit{Dirichlet form} $Q_{\gamma_{-1}}$ associated with the inverse Gauss measure is defined by
$$
Q_{\gamma_{-1}}(f)
:= \frac{1}{2}
\int_{\Rn} |\nabla f(x)|^2 \, \gamma_{-1}(x) dx,
\quad f \in C_c^\infty(\Rn),
$$
where as usual $C_c^\infty(\Rn)$ denotes the space of smooth functions with compact support in $\Rn$. $Q_{\gamma_{-1}}$ defines the operator
$$
\mathcal{A}(f)
:= - \frac{1}{2} \Delta f - x \cdot \nabla f,
\quad f \in C_c^\infty(\Rn).
$$
The Dirichlet form $Q_{\gamma_{1}}$ for the Gauss measure defines the \textit{Ornstein-Uhlenbeck} operator,
$$
\mathcal{L}(f)
:= - \frac{1}{2} \Delta f + x \cdot \nabla f,
\quad f \in C_c^\infty(\Rn).
$$

The harmonic analysis for the Ornstein-Uhlenbeck operator was firstly studied by B. Muckenhoupt (\cite{Mu}) and ever since has been an active work area (see for instance
\cite{
AFS,
Br1,
CMM,
DS,
FGS,
FHS,
FS,
GCMMST1,
GCMMST2,
GCMST2,
GCMST1,
Gu,
GST,
HMM,
Ke,
MMS,
MPS,
Me,
Pe,
PS,
Po,
Sj1,
Ur1}).
Recently, W. Urbina (\cite{Urb2}) has published a monograph about Gaussian harmonic analysis.\\

The connection between the Gaussian measure and the Ornstein-Uhlenbeck operator motivated F. Salogni (\cite{Sa}) to begin the study of harmonic analysis in the inverse Gaussian measure setting. The semigroup of operators
$\{T_t^\mathcal{A}:=e^{-t \mathcal{A}}\}_{t>0}$ generated by the operator $-\mathcal{A}$ is a symmetric diffusion semigroup in the sense of E. Stein (\cite{StLP}). $L^p(\Rn,\gamma_{-1})$-boundedness properties of the maximal operator $T_*^\mathcal{A}$ defined by
$$
T_*^{\mathcal{A}}(f)
:= \sup_{t>0} |T_t^\mathcal{A}(f)|,
$$
were established in \cite[Theorem 3.3.6]{Sa}.
In \cite{Br2}, T. Bruno introduces a new Hardy space $H^1$ and obtains endpoint ($p=1$) results for the imaginary powers $\mathcal{A}^{i \gamma}$, $\gamma \in \mathbb{R}$, and also for certain Riesz transforms. The results concerning Riesz transforms associated with $\mathcal{A}$ were completed in
\cite[Theorem 1.1]{BrSj}.\\

We now recall some definitions and results about the operator $\mathcal{A}$ and the semigroups of operators generated by $-\mathcal{A}$ and $-\sqrt{\mathcal{A}}$ that we will need in the sequel.\\

For every $k \in \NN$ we denote by $H_k$ the $k$-th \textit{Hermite polynomial} defined by
$$
H_k(z)
:=
(-1)^k e^{z^2}\frac{d^k}{dz^k}e^{-z^2},
\quad z \in \mathbb{R}.$$
For every $k=(k_1, \dots, k_n) \in \NN^n$ we define
$$
H_k(x)
:=
\prod_{i=1}^n H_{k_i}(x_i),
\quad x=(x_1, \dots, x_n) \in \Rn,$$
and
$$
\widetilde{H}_k(x)
:=
\gamma_1(x) H_k(x).$$
We have that, for every $k=(k_1, \dots, k_n) \in \NN^n$,
$$
\mathcal{A}\widetilde{H}_k
=
(k_1 + \dots + k_n + n) \widetilde{H}_k.
$$
We define the operator $\widetilde{\mathcal{A}}$
as follows
$$
\widetilde{\mathcal{A}}f
:=
\sum_{k=(k_1, \dots, k_n) \in \NN^n}
c_k(f) (k_1 + \dots + k_n + n) \widetilde{H}_k,
\quad f \in D(\widetilde{\mathcal{A}}),$$
where
$$
c_k(f)
:= \frac{1}{\|\widetilde{H}_k\|_{L^2(\Rn,\gamma_{-1})}^2}
\int_{\Rn} f(x) \widetilde{H}_k(x) \, \gamma_{-1}(x) dx
$$
and
$$
D(\widetilde{\mathcal{A}})
:= \Big\{
f \in L^2(\Rn,\gamma_{-1}) \text{ : }
\sum_{k=(k_1, \dots, k_n) \in \NN^n}
|c_k(f)|^2  (k_1 + \dots + k_n + n)^2
\|\widetilde{H}_k\|_{L^2(\Rn,\gamma_{-1})}^2
< \infty
\Big\}.
$$
We can see that
$C^\infty_c(\Rn) \subset D(\widetilde{\mathcal{A}})$
and that
$$
\mathcal{A}f
=
\widetilde{\mathcal{A}}f,
\quad f \in C^\infty_c(\Rn).
$$
With a slight abuse of notation, we identify
$\mathcal{A}$ with $\widetilde{\mathcal{A}}$.\\

By
$\{T_t^\mathcal{A}\}_{t>0}$
we denote the semigroup of operators generated by $-\mathcal{A}$. For every $t>0$ and $1 \leq p \leq \infty$ the operator $T_t^\mathcal{A}$ admits the following integral representation
$$
T_t^\mathcal{A}(f)(x)
:=
\int_{\Rn} T_t^\mathcal{A}(x,y) f(y) \, dy,
\quad x \in \Rn,
$$
for every $f \in L^p(\Rn,\gamma_{-1})$, where
$$
T_t^\mathcal{A}(x,y)
:=
\frac{e^{-nt}}{\pi^{n/2}(1-e^{-2t})^{n/2}}
\exp\Big( \frac{-|x-e^{-t}y|^2}{1-e^{-2t}} \Big), \quad x, y \in \Rn, \quad t>0.$$
$\{T_t^\mathcal{A}\}_{t>0}$ is a symmetric diffusion semigroup in $L^p(\Rn,\gamma_{-1})$,
$1 \leq p \leq \infty$.\\

By using the subordination formula,
if $\{P_t^\mathcal{A}\}_{t>0}$ denotes the semigroup of operators generated by $-\sqrt{\mathcal{A}}$ we have that
\begin{equation}\label{eq:subordination}
P_t^\mathcal{A}(f)
:=
\frac{t}{2\sqrt{\pi}}
\int_0^\infty
\frac{ e^{-t^2/4u}}{u^{3/2}}
T_u^\mathcal{A}(f)
\, du,
\quad t>0.
\end{equation}

\quad \\
For every $k \in \NN$ we consider the
\textit{maximal operators}
$$
T_{*,k}^\mathcal{A}(f)
:=
\sup_{t>0}
| t^k \partial_t^k T_t^\mathcal{A}(f)|
$$
and
$$
P_{*,k}^\mathcal{A}(f)
:=
\sup_{t>0}
| t^k \partial_t^k P_t^\mathcal{A}(f)|.
$$
Since $\{T_t^\mathcal{A}\}_{t>0}$ is a diffusion semigroup, according to \cite[Corollary 4.2]{LeXu}, for every $k \in \NN$, the maximal operators
$T_{*,k}^\mathcal{A}$
and
$P_{*,k}^\mathcal{A}$
are bounded from $L^p(\Rn,\gamma_{-1})$ into itself, for each $1<p<\infty$.
Furthermore, the ideas in \cite[pp. 472--473]{LiSj} imply that $P_{*,k}^\mathcal{A}$ is also bounded from
$L^1(\Rn,\gamma_{-1})$ into
$L^{1,\infty}(\Rn,\gamma_{-1})$, for every $k \in \mathbb{N}$. Salogni (\cite[Theorem 3.3.6]{Sa}) proved that $T_{*}^\mathcal{A}=T_{*,0}^\mathcal{A}$ is bounded from $L^1(\Rn,\gamma_{-1})$ into
$L^{1,\infty}(\Rn,\gamma_{-1})$. 

\begin{Rem*}
As a consequence of our
results (see Theorem \ref{Th:1.1} below) we can deduce that $T_{*,1}^\mathcal{A}$ is bounded from $L^1(\Rn,\gamma_{-1})$ into
$L^{1,\infty}(\Rn,\gamma_{-1})$. At this moment we do not know if $T_{*,k}^\mathcal{A}$ is bounded from $L^1(\Rn,\gamma_{-1})$ into
$L^{1,\infty}(\Rn,\gamma_{-1})$, when $k \geq 2$. As far as we know in the Ornstein-Uhlenbeck setting the question is also open for $k \geq 1$.
\end{Rem*}

Also we consider the
\textit{centered Hardy-Littlewood maximal function} associated with the inverse Gaussian measure
$\mathcal{M}_{\gamma_{-1}}$ defined by
$$
\mathcal{M}_{\gamma_{-1}}(f)(x)
:=
\sup_{r>0} \frac{1}{\gamma_{-1}(B(x,r))}
\int_{B(x,r)} |f(y)| \, \gamma_{-1}(y) dy,
\quad x \in \Rn.
$$
In this paper we characterize the Banach lattices with the Hardy-Littlewood property in terms of the $L^p(\Rn,\gamma_{-1})$-boundedness properties 
of the maximal function
$\mathcal{M}_{\gamma_{-1}}$ and the maximal operators defined by the semigroups
$\{T_t^\mathcal{A}\}_{t>0}$
and
$\{P_t^\mathcal{A}\}_{t>0}$
(Theorems \ref{Th:1.3}, \ref{Th:1.1} and \ref{Th:1.2}).
Our study is motivated by \cite{HTV1} in the Gaussian context.\\

Before stating our results we need to recall some definitions and properties about Banach lattices and the Hardy-Littlewood property.

\begin{DefBL}
A \textit{Banach lattice} is a real Banach space $X$ endowed with an order relation $\leq$ satisfying the following properties. Let $x,y,z \in X$ and $a \in \mathbb{R}$, $a \geq 0$,\\
\begin{itemize}
\item[$(i)$]
$x \leq y
\, \Rightarrow \,
x+z \leq y+z$;\\

\item[$(ii)$]
$ax \geq 0$ provided that $x \geq 0$ in $X$;\\

\item[$(iii)$]
there exist the least upper bound $\sup\{x,y\}$ and the greatest lower bound $\inf\{x,y\}$ for $x$ and $y$;\\

\item[$(iv)$] if we define
$|w|
:= \sup\{-w,w\}$, $w \in X$,
$\|x\| \leq \|y\|$ provided that $|x| \leq |y|$.\\
\end{itemize}
\end{DefBL}

\begin{DefKFS}
Let $(\Omega,\Sigma,\mu)$ be a complete $\sigma$-finite measure space. A Banach space $X$ consisting of equivalence classes, modulo equality almost everywhere, of locally integrable, real valued functions in $\Omega$ is said to be a \textit{K\"othe function space} when the following properties hold\\
\begin{itemize}
\item[$(\alpha)$] If $|f(w)| \leq |g(w)|$, a.e. $w \in \Omega$, with $f$ measurable and $g \in X$, then $f \in X$ and $\|f\| \leq \|g\|$.\\

\item[$(\beta)$] For every $E \in \Sigma$ such that $\mu(E)<\infty$, $\chi_E \in X$.\\
\end{itemize}
\end{DefKFS}

Here $\chi_E$ denotes the characteristic function over the set $E$.\\

Every K\"othe space is a Banach lattice with the obvious order,
$$
f \geq 0
\quad \text{when}\quad
f(w) \geq 0, \quad \text{a.e. } w \in \Omega.$$
This lattice is $\sigma$-order complete.
Moreover, each order continuous Banach lattice with weak unit is order isometric to a K\"othe function space
(\cite[Theorem 1.b.14]{LT}). Then, a separable Banach lattice is order isometric to a K\"othe function space if, and only if, it is $\sigma$-order complete. This fact justifies why we establish our results for K\"othe function spaces. If $X$ is a K\"othe function space, we denote by $X'$ the space  of all the integrals in $X$ (\cite[p. 29]{LT}).

\begin{DefHLP}
Let $n \in \mathbb{N}$ and $X$ be a Banach lattice.
Assume that $J$ is a finite subset of rational positive numbers. If $f : \Rn \longrightarrow X$ is a locally integrable function, we define
$$
\mathcal{M}^X_J(f)(x)
:=
\sup_{r \in J} \frac{1}{|B(x,r)|}
\int_{B(x,r)} |f(y)| dy.
$$
A Banach lattice $X$ is said to have the
\textit{Hardy-Littlewood property} if there exists $1<p<\infty$ such that, for every finite subset $J$ of $\mathbb{Q}_+$, $\mathcal{M}^X_J$ is bounded from
$L^p(\Rn,X)$ into itself and
$$
\sup_{\substack{
J \subset \mathbb{Q}_+\\
J \text{ finite}}}
\|\mathcal{M}^X_J\|_{L^p(\Rn,X) \to L^p(\Rn,X)}
< \infty.
$$
\end{DefHLP}

The Hardy-Littlewood property for a Banach lattice does not depend on the dimension $n$.\\

For a general Banach lattice $X$ the supremum must be taken over finite sets $J$. When $X$ is a K\"othe function space this restriction is not needed. Suppose that $X$ is a  K\"othe function space over
$(\Omega,\Sigma,\mu)$. Every function
$f : \Rn \longrightarrow X$ is understood as a two variables function
$(x,w) \in \Rn \times \Omega \longmapsto f(x,w) \in \mathbb{R}$. We define now the 
\textit{centered Hardy-Littlewood maximal function} with respect to the first variable as follows
\begin{equation}\label{MaxHL}
\mathcal{M}^X(f)(x,w)
:=
\sup_{r >0} \frac{1}{|B(x,r)|}
\int_{B(x,r)} |f(y,w)| dy,
\quad x \in \Rn \text{ and } w \in \Omega,
\end{equation}
where $f : \Rn \longrightarrow X$ is locally integrable.
J. Bourgain (\cite{Bo}) proved that a K\"othe function space $X$ has the UMD property if, and only if, $\mathcal{M}^X$ is bounded from $L^p(\Rn,X)$ into itself and from $L^{p'}(\Rn,X^*)$ into itself, for some
(equivalently, for every) $1<p<\infty$, where $X^*$ denotes the dual space of $X$ and $p'=p/(p-1)$.

\begin{DefFP}
It is said that a K\"othe function space has the
\textit{Fatou property} when the following property holds. If, for every $n \in \mathbb{N}$, $f_n \in X$ and $f_n(w)\geq 0$, a.e. $w \in \Omega$, satisfying that
$f_n(w) \uparrow f(w)$, a.e. $w \in \Omega$, as $n \to \infty$, and $\sup_{n \in \mathbb{N}} \|f_n\|_X<\infty$, then $f \in X$ and $\|f\|_X = \lim_{n \to \infty} \|f_n\|_X$.
\end{DefFP}

By using Lebesgue's monotone convergence theorem we can see that a K\"othe function space $X$ having the Fatou property has the Hardy-Littlewood property if, and only if, there exists $1<p_0<\infty$ such that $\mathcal{M}^X$ is bounded from $L^{p_0}(\Rn,X)$ into itself. In \cite[Theorem 1.7]{GCMT1} it was proved that this condition is equivalent to the boundedness of $\mathcal{M}^X$ from $L^p(\Rn,X)$ into itself, for every $1<p<\infty$, and to the boundedness of $\mathcal{M}^X$ from $L^1(\Rn,X)$ into $L^{1,\infty}(\Rn,X)$. \\

The Hardy-Littlewood property for Banach lattices has been investigated in \cite{GCMT1}, \cite{GCMT2} and \cite{HMST} (see also \cite{RbF}). In \cite{HTV1} the Hardy-Littlewood property for K\"othe function spaces was characterized by using maximal operators associated with Ornstein-Uhlenbeck operators with respect to the Gauss measure. In this paper we obtain new characterizations of the K\"othe function spaces with the Hardy-Littlewood property via maximal operators associated with the semigroups
$\{T_t^\mathcal{A}\}_{t>0}$
and
$\{P_t^\mathcal{A}\}_{t>0}$ (Theorems \ref{Th:1.1} and \ref{Th:1.2}).\\

Suppose that $g : (0,\infty) \longrightarrow \mathbb{R}$ is $C^m(0,\infty)$ where $m \in \mathbb{N}$ and that
$m-1 \leq \alpha < m$.
The \textit{Weyl fractional derivative} of order $\alpha$,
$D^\alpha g$, is defined by
$$
D^\alpha g(t)
:=
\frac{1}{\Gamma(m-\alpha)}
\int_0^\infty g^{(m)}(t+s) s^{m-\alpha-1} \, ds,
\quad t \in (0,\infty).
$$
If $X$ is a K\"othe function space and $f \in L^p(\Rn,X)$, $1 \leq p < \infty$, we define
$$
\mathcal{M}^X_{\gamma_{-1}}(f)(x,w)
:=
\sup_{r>0} \frac{1}{\gamma_{-1}(B(x,r))}
\Big|
\int_{B(x,r)} f(y,w) \, \gamma_{-1}(y) dy \Big|,
\quad x \in \Rn, \quad w \in \Omega,
$$
and, for every $\alpha>0$,
$$
T_{*,\alpha}^\mathcal{A}(f)(x,w)
:=
\sup_{t>0}
| t^\alpha \partial_t^\alpha T_t^\mathcal{A}(f(\cdot,w))(x)|,
\quad x \in \Rn, \quad w \in \Omega,
$$
and
$$
P_{*,\alpha}^\mathcal{A}(f)(x,w)
:=
\sup_{t>0}
| t^\alpha \partial_t^\alpha P_t^\mathcal{A}(f(\cdot,w))(x)|,
\quad x \in \Rn, \quad w \in \Omega.
$$

\quad \\
From now on, when it is understood from the context, we will use the notation $S(f)(x):=S(f)(x,\cdot)$, $x\in\Rn,$ for any operator $S$ appearing along the paper.

We are finally in position of stating our results.

\begin{Th}\label{Th:1.3}
Let $X$ be a K\"othe function space with the Fatou property. The following assertions are equivalent.
\begin{itemize}
\item[$(a)$]
$X$ has the Hardy-Littlewood property.

\item[$(b)$]
$\mathcal{M}^X_{\gamma_{-1}}$ is bounded from
$L^1_X(\Rn,\gamma_{-1})$ into
$L^{1,\infty}_X(\Rn,\gamma_{-1})$.

\item[$(c)$]
For every $f \in L^1_X(\Rn,\gamma_{-1})$,
$\mathcal{M}^X_{\gamma_{-1}}f(x) \in X$ for almost all $x \in \Rn$.\\
\end{itemize}
\end{Th}

\begin{Th}\label{Th:1.1}
Let $X$ be a K\"othe function space  with the Fatou property, $1<p<\infty$ and $\alpha \in [0,1]$. Consider the following assertions.
\begin{itemize}
\item[$(a)$]
$X$ has the Hardy-Littlewood property.

\item[$(b_\alpha)$]
$T_{*,\alpha}^\mathcal{A}$ is bounded from
$L^p_X(\Rn,\gamma_{-1})$ into itself.

\item[$(c_\alpha)$]
$T_{*,\alpha}^\mathcal{A}$ is bounded from
$L^1_X(\Rn,\gamma_{-1})$ into
$L^{1,\infty}_X(\Rn,\gamma_{-1})$.

\item[$(d_\alpha)$]
For every $f \in L^p_X(\Rn,\gamma_{-1})$,
$T_{*,\alpha}^\mathcal{A}(f)(x) \in X$ for almost all $x \in \Rn$.

\item[$(e_\alpha)$]
For every $f \in L^1_X(\Rn,\gamma_{-1})$,
$T_{*,\alpha}^\mathcal{A}(f)(x) \in X$ for almost all $x \in \Rn$.
\end{itemize}
Then, when
\begin{itemize}
\item $\alpha=0$,
$(a) \Leftrightarrow (b_\alpha) \Leftrightarrow (c_\alpha) \Leftrightarrow (d_\alpha) \Leftrightarrow (e_\alpha)$;

\item $\alpha \in (0,1]$,
$(a) \Rightarrow (b_\alpha), \: (c_\alpha), \:(d_\alpha) \text{ and } (e_\alpha)$.\\




\end{itemize}
\end{Th}

\begin{Th}\label{Th:1.2}
Let $X$ be a K\"othe function space with the Fatou property, $1<p<\infty$ and $\alpha \in [0,\infty)$. Consider the following assertions.
\begin{itemize}
\item[$(a)$]
$X$ has the Hardy-Littlewood property.

\item[$(b_\alpha)$]
$P_{*,\alpha}^\mathcal{A}$ is bounded from
$L^p_X(\Rn,\gamma_{-1})$ into itself.

\item[$(c_\alpha)$]
$P_{*,\alpha}^\mathcal{A}$ is bounded from
$L^1_X(\Rn,\gamma_{-1})$ into
$L^{1,\infty}_X(\Rn,\gamma_{-1})$.

\item[$(d_\alpha)$]
For every $f \in L^p_X(\Rn,\gamma_{-1})$,
$P_{*,\alpha}^\mathcal{A}(f)(x) \in X$ for almost all $x \in \Rn$.

\item[$(e_\alpha)$]
For every $f \in L^1_X(\Rn,\gamma_{-1})$,
$P_{*,\alpha}^\mathcal{A}(f)(x) \in X$ for almost all $x \in \Rn$.
\end{itemize}
Then, when
\begin{itemize}
\item $\alpha=0$,
$(a) \Leftrightarrow (b_\alpha) \Leftrightarrow (c_\alpha) \Leftrightarrow (d_\alpha) \Leftrightarrow (e_\alpha)$;

\item $\alpha>0$,
$(a) \Rightarrow (b_\alpha), \: (c_\alpha), \:(d_\alpha) \text{ and } (e_\alpha)$.\\




\end{itemize}
\end{Th}
Some comments related to Theorems \ref{Th:1.1} and \ref{Th:1.2} are in order. Suppose that $\{\mathbb{T}_t\}_{t>0}$ is a symmetric diffusion semigroup on $(\Lambda,\mu)$ and $\{\mathbb{P}_t\}_{t>0}$ is the corresponding subordinated Poisson semigroup. H. Li and P. Sj\"ogren (\cite[p. 473]{LiSj}) proved that, for every $k\in \mathbb{N}$, there exists $C>0$ such that
$$
\mathbb{P}_{*,k}(f)\le C\sup_{t>0}\Big|\frac{1}{t}\int_0^t\mathbb{T}_s(f)ds\Big|.
$$
Then by using Hopf-Dunford-Schwartz ergodic theorem it follows that $\mathbb{P}_{*,k}$ is bounded from $L^1(\Lambda)$ into $L^{1,\infty}(\Lambda)$, for every $k\in \mathbb{N}$.\\

Let $X$ be a K\"othe function space on $(\Omega,\nu)$. We define the 
\textit{maximal ergodic function} associated to $\{\mathbb{T}_t\}_{t>0}$ by
$$
\mathbb{M}(f)(x,w)
:=\sup_{t>0}\frac{1}{t}\Big|\int_0^t\mathbb{T}_s(f(\cdot,w))(x)ds\Big|,\,\,\,x\in \Lambda\,\,\, and\,\,\,w\in \Omega.
$$
A Banach space $Y$ is said to have the \textit{UMD property}
when the Hilbert transform can be extended to $L^p(\mathbb{R})\bigotimes Y$, in the obvious way, as a bounded operator from  $L^p(\mathbb{R})\bigotimes Y$ into itself with the topology of $L^p_Y(\mathbb{R})$ for some (equivalently, for every) $1<p<\infty$. The interested reader can find in \cite[Chapter 4]{Hyt} the main properties of UMD Banach spaces.\\

In \cite[Theorem 1]{Xu} Q. Xu proved that if $X$ is a UMD K\"othe function space, then $\mathbb{M}$ defines a bounded operator from $L^p_X(\Lambda)$ into itself, for every $1<p<\infty$. Hence, if $X$ is a UMD K\"othe function space, for every $k\in \mathbb{N}$ and $1<p<\infty$, $\mathbb{P}_{*,k}$ is bounded from $L^p_X(\Lambda)$ into itself. Furthermore, as Xu commented (\cite[Problem 10]{Xu}), it is not known if, even when $X$ is a UMD K\"othe function space, $\mathbb{M}$ is bounded from $L^1_X(\Lambda)$ into $L^{1,\infty}_X(\Lambda)$.

\quad\\
We prove our results by using a method introduced by Muckenhoupt
(\cite{Mu}) that consists on decomposing the operator under consideration in two  parts that are called \textit{local} and \textit{global} operators, respectively. In order to see the
$L^p(\Rn,\gamma_{-1})$--boundedness properties of the local operators we need the corresponding for the classical operators, that is, those operators associated with the Euclidean heat and Poisson semigroups. Next, we state these results that we will be needed, which are also interesting by themselves.\\

We consider the \textit{Euclidean heat semigroup} (i.e., the one associated to $-\frac{1}{2}\Delta$), $\{W_t\}_{t>0}$, given by
$$
W_t(f)(x)
:=
\int_{\Rn} W_t(x-y) f(y) \, dy,
\quad x \in \Rn, \quad t>0,$$
where
$$
W_t(z)
:=
\frac{1}{(2\pi t)^{n/2}} e^{-|z|^2/2t},
\quad z \in \Rn, \quad t>0,
$$
and the \textit{Euclidean Poisson semigroup}
$\{P_t\}_{t>0}$ defined by
$$
P_t(f)(x)
:=
\int_{\Rn} P_t(x-y) f(y) \, dy,
\quad x \in \Rn, \quad t>0,$$
with
$$
P_t(z)
:=
\frac{\Gamma((n+1)/2)}{\pi^{(n+1)/2}}
\frac{t}{(t^2+|z|^2)^{(n+1)/2}},
\quad z \in \Rn, \quad t>0.
$$

\quad\\
If $X$ is a K\"othe function space, $\alpha>0$, $1 \leq p < \infty$ and $f \in L^p_X(\Rn,{dx})$, we introduce the following maximal operators
$$
W_{*,\alpha}(f)(x,w)
:=
\sup_{t>0}
| t^\alpha \partial_t^\alpha W_t(f(\cdot,w))(x)|,
\quad x \in \Rn, \quad w \in \Omega,
$$
and
$$
P_{*,\alpha}(f)(x,w)
:=
\sup_{t>0}
| t^\alpha \partial_t^\alpha P_t(f(\cdot,w))(x)|,
\quad x \in \Rn, \quad w \in \Omega.
$$

\begin{Th}\label{Th:1.4}
Let $X$ be a K\"othe function space with the Fatou property, $1<p<\infty$ and $\alpha \in [0,\infty)$.
Denote by $S_{*,\alpha}$ the maximal operators
$W_{*,\alpha}$ or $P_{*,\alpha}$ and
consider the following assertions.
\begin{itemize}
\item[$(a)$]
$X$ has the Hardy-Littlewood property.

\item[$(b_\alpha)$]
$S_{*,\alpha}$ is bounded from
$L^p_X(\Rn,dx)$ into itself.

\item[$(c_\alpha)$]
$S_{*,\alpha}$ is bounded from
$L^1_X(\Rn,dx)$ into
$L^{1,\infty}_X(\Rn,dx)$.

\item[$(d_\alpha)$]
For every $f \in L^p_X(\Rn,dx)$,
$S_{*,\alpha}(f)(x) \in X$ for almost all $x \in \Rn$.

\item[$(e_\alpha)$]
For every $f \in L^1_X(\Rn,dx)$,
$S_{*,\alpha}(f)(x) \in X$ for almost all $x \in \Rn$.
\end{itemize}
Then, we have that
\begin{itemize}
\item[$(i)$]
$(a) \Leftrightarrow
(b_\alpha) \Leftrightarrow
(c_\alpha)  \Leftrightarrow
(d_\alpha)  \Leftrightarrow
(e_\alpha) $,
for $\alpha=0$.

\item[$(ii)$]
$(a) \Rightarrow
(b_\alpha), 
(c_\alpha), 
(d_\alpha) \text{ and }
(e_\alpha) $,
for every $\alpha \geq 0$.

\end{itemize}
\end{Th}

Observe that in the statements of our results we have imposed $X$ to have the Fatou property. This is because at some point in the proofs  we compare the operators with the classical Hardy- Littlewood maximal function $\mathcal{M}^X$. In \cite{HTV1} the authors use a different definition of the  Hardy-Littlewood property for a K\"othe function space. They say that a K\"othe function space $X$ satisfies the Hardy-Littlewood property if, and only if, there exists $1<p_0<\infty$ such that $\mathcal{M}^X$ is bounded from $L^{p_0}(\Rn,X)$ into itself (and the equivalences that follow). Therefore, if we had used this definition in our results then  they would remain valid and the Fatou property would be needed only for the implications that concern $(d_\alpha)$ and $(e_\alpha)$ in Theorems \ref{Th:1.1} and  \ref{Th:1.2}.

\,

As mentioned before, our study is motivated by \cite{HTV1}, which was developed in the Gaussian context.
We remark that in \cite{HTV1} the authors considered only the maximal operator for the Poisson semigroup associated with the Ornstein-Uhlenbeck operator, which is in our notation the maximal operator $P^{\mathcal{L}}_{*,0}$. On the other hand, note that the heat and Poisson kernels are positive, but this property is lost after taking derivatives.\\

In the following sections we present the proofs of our theorems. Throughout the paper $C$ and $c$ always denote positive constants that might change in each appearance.
We also write $a\lesssim b$ as shorthand for $a\leq Cb$ and moreover will use the notation $a\approx b$ if $a\lesssim b$ and $b\lesssim a$.

\section{Proof of Theorem \ref{Th:1.3}}

Firstly we introduce some notation that will be useful throughout the paper.\\

For some $\beta>0$ we divide $\Rn \times \Rn$ in two regions, namely the
\textit{local region}
$$
N_\beta
:= \Big\{(x,y) \in \Rn \times \Rn \text{ : }
|x-y| \leq \beta 
\Big(1 \land \frac{1}{|x|}\Big)
\Big\} 
$$
and the \textit{global region}
$$
N_\beta^c
:= \Big\{(x,y) \in \Rn \times \Rn \text{ : }
|x-y| > \beta 
\Big(1 \land \frac{1}{|x|}\Big)
\Big\}.
$$
Here $a \land b := \min\{a,b\}$, $a, b \in \mathbb{R}$.
If $T$ represents any of the operators defined above, we introduce
$$
T_{loc(\beta)}(f)(x)
:= T( \chi_{N_\beta}(x,\cdot)f)(x),
\quad x \in \Rn,
$$
and
$$
T_{glob(\beta)}(f)(x)
:= T( \chi_{N_\beta^c}(x,\cdot)f)(x),
\quad x \in \Rn.
$$
The precise value of $\beta$ will be clear in each occurrence. When $\beta=1$ we do not write it. \\

In order to prove Theorem \ref{Th:1.3} we will adapt to our inverse Gaussian setting some ideas from \cite[Section 2]{HTV2}.\\

We start with the following useful lemma.

\begin{Lem}\label{Lem:Q}
There exists a measurable function $Q$ on $\Rn \times \Rn$ such that $Q(x,y)\geq 0$, $x,y \in \Rn$. Moreover,
\begin{itemize}
\item[$(i)$] the operator $\mathbb{Q}$ defined by
$$
\mathbb{Q}(g)(x)
:=
\int_{\Rn} Q(x,y) g(y) \, dy
$$
is bounded from $L^1(\Rn,\gamma_{-1})$ into itself;\\

\item[$(ii)$] for every $g \in L^1(\Rn,\gamma_{-1})$, $g \geq 0$,
$$\mathcal{M}^X_{\gamma_{-1},glob}(g)
\leq \mathbb{Q}(g).$$
\end{itemize}
\end{Lem}

\begin{Rem}\label{Rem:Mglob}
Lemma \ref{Lem:Q} implies that $\mathcal{M}_{\gamma_{-1},glob}^X$ is bounded from $L^1_X(\Rn,\gamma_{-1})$ into itself. This property does not depend on the Hardy-Littlewood property for the K\"othe function space $X$.
\end{Rem}

\begin{proof}[Proof of Lemma \ref{Lem:Q}]
We define
$$
Q(x,y)
:=
\sup_{r>0}
\frac{1}{\gamma_{-1}(B(x,r))}
\chi_{N^c}(x,y) \, \chi_{B(x,r)}(y) \, e^{|y|^2},
\quad x,y \in \Rn.$$
It is clear that
$$
Q(x,y)
\leq
\frac{1}{\gamma_{-1}(B(x,|x-y|))}
\chi_{N^c}(x,y) \, e^{|y|^2},
\quad x,y \in \Rn,
$$
and that
$$
\mathcal{M}^X_{\gamma_{-1},glob}(g)
\leq
\int_{\Rn} Q(x,y) g(y) \, dy,
\quad x \in \Rn,
$$
for every Lebesgue measurable function $g \geq 0$ in $\Rn$.\\

It remains to justify property $(i)$.
Suppose firstly that $n \geq 2$.
Let $x \in \Rn \setminus \{0\}$ and $R>0$.
Since $\gamma_{-1}$ is rotation invariant we can assume that
$x=|x|e_n$. We have that
\begin{align*}
\gamma_{-1}(B(x,R))
& = \int_{|u|<R} e^{|u+x|^2} \, du 
 \geq e^{|x|^2+R^2/4}
\int_{\varphi_0}^{\pi/2} \int_{R/2}^R r^{n-1} e^{2r|x|\cos \varphi_1}
\sin \varphi_1 \, dr \, d\varphi_1 \\
& = C
\frac{e^{|x|^2+R^2/4} }{|x|}
\int_{R/2}^R r^{n-2} \int_0^{2r|x|\cos \varphi_0}
 e^u \, du \, dr,
\end{align*}
where $\varphi_0:=\pi/3$.
Take
$R|x|\cos \varphi_0 \geq \beta$, for some $\beta>0$. Since
\begin{align*}
\int_{R/2}^R \int_0^{2r|x|\cos \varphi_0}
 e^u \, du \, dr
&\geq \int_{R/2}^R\int_{r|x|cos\varphi_0}^{2r|x|cos\varphi}e^ududr
\geq \int_{R/2}^Re^{r|x|cos\varphi_0}r|x|cos\varphi_0dr\\
&\geq\frac{R\beta}{4}
e^{(R|x|cos\varphi_0)/2} 
\geq\frac{\beta^2}{4|x|cos\varphi_0}
e^{(R|x|\cos\varphi_0)/2},
\end{align*}
we get
$$
\gamma_{-1}(B(x,R))
\gtrsim \frac{e^{|x|^2}e^{R^2/4}R^{n-2}}{|x|^2}
e^{(R|x| \cos \varphi_0)/2}.
$$

\quad\\
Next, we distinguish several cases.\\
\begin{itemize}
\item[$(a)$] Assume that $|x|\geq 1$, or equivalently, 
$ 1 \land 1/|x| =1/|x|$. 
Then,
$$|x-y| \, |x| \cos \varphi_0
\geq \cos \varphi_0,
\quad (x,y) \in N^c.$$
It follows that
\begin{align*}
Q(x,y)
& \lesssim \,
e^{|y|^2-|x|^2} \,
e^{-|x-y|^2/4} \,
e^{- (|x-y| \, |x| \cos \varphi_0)/2} \,
|x|^2 \,
|x-y|^{2-n} \\
& \lesssim  \,
e^{|y|^2-|x|^2} \,
e^{-|x-y|^2/4} \,
e^{- (|x-y| \, |x| \cos \varphi_0)/2} \,
|x|^n.
\end{align*}

\quad 

\begin{itemize}
\item[$(a.1)$]
We consider $|x-y|\leq |x|/2$.
Then, $|x|/2 \leq |y| \leq 3|x|/2$.
We get,
$$
Q(x,y)
\lesssim  \,
e^{|y|^2-|x|^2} \,
e^{-( |x-y| \, |y| \cos \varphi_0)/3} \,
|y|^n.
$$

\quad 

\item[$(a.2)$]
Suppose $|x-y|>|x|/2$.
Since $|x|\geq 1$, $|x-y|>1/2$ and it follows that
$$
\gamma_{-1}(B(x,|x-y|))
\geq
\gamma_{-1}(B(x,1/2))
\gtrsim
e^{|x|^2 + (|x| \cos \varphi_0)/4}
|x|^{-2}.
$$
Then,
$$
Q(x,y)
\lesssim  e^{|y|^2-|x|^2 -(|x| \cos \varphi_0)/4}
|x|^{2}.
$$
\end{itemize}

\quad\\
\item[$(b)$] Assume that $|x| \leq 1$, that is, $1 \land 1/|x| =1$. Then,
$$
B(x,1)
\subset
B(x,|x-y|),
\quad (x,y) \in N^c,$$
and for certain $C>0$,
$$
\gamma_{-1}(B(x,|x-y|))
\geq
\gamma_{-1}(B(x,1))
=
\int_{B(x,1)} e^{|y|^2} \, dy
\geq C,
\quad (x,y) \in N^c,
$$
where in the last inequality we have used that $e^{|y|^2}\ge 1$, $y\in \Rn$, so
$\gamma_{-1}(B(x,1))\ge |B(x,1)|=|B(0,1)|$, $x\in \Rn$.
Thus, we have in this case that
$$
Q(x,y)
\lesssim
e^{|y|^2}.$$
\end{itemize}

\quad\\
Let now $g \in L^1(\Rn,\gamma_{-1})$.
By taking into account the above estimates we obtain
\begin{align*}
\|\mathbb{Q}(g)\|_{L^1(\Rn,\gamma_{-1})}
& \lesssim
\int_{|x|\geq 1}
e^{|x|^2}
\int_{|x-y|\leq |x|/2}
e^{|y|^2-|x|^2} \,
e^{-(|x-y| \, |y| \cos \varphi_0)/3} \,
|g(y)| \,
|y|^n \,
dy \, dx \\
& \quad
+
\int_{|x|\geq 1}
e^{|x|^2}
\int_{|x-y|> |x|/2}
e^{|y|^2-|x|^2 - (|x| \cos \varphi_0)/4} \,
|g(y)| \,
|x|^2 \,
dy \, dx \\
& \quad
+
\int_{|x|\leq 1}
e^{|x|^2}
\int_{\Rn}
e^{|y|^2} \,
|g(y)| \,
 \,
dy \, dx
 \\
& \lesssim
\|g\|_{L^1(\Rn,\gamma_{-1})}.
\end{align*}

\quad\\
Finally, suppose that $n=1$.
Let $x \in \mathbb{R}\setminus \{0\}$ and $R>0$.
We can assume that $x>0$.
We have that
\begin{align*}
\gamma_{-1}(B(x,R))
& =
\int_{|u|<R}
e^{x^2+u^2+2xu} \, du
 \geq
e^{x^2 + R^2/4}
\int_{R/2<u<R}
e^{2xu} \, du\\
&\gtrsim
\frac{e^{x^2 + R^2/4+Rx} }{x},
\end{align*}
provided that $Rx \geq \beta$, for certain $\beta>0$. Then,
$$
\gamma_{-1}(B(x,|x-y|))
\gtrsim
\frac{e^{x^2 + |x-y|^2/4+|x-y| \, |x|} }{|x|},
\quad (x,y) \in N^c.
$$
By proceeding as above we can see that the operator $\mathbb{Q}$ is bounded from
$L^1(\mathbb{R},\gamma_{-1})$ into itself.
\end{proof}

We now study the operator $\mathcal{M}^X_{\gamma_{-1},loc}$.
Let $f \in L^1_X(\Rn,\gamma_{-1})$.
We have that
$$
\mathcal{M}^X_{\gamma_{-1},loc}(f)(x,w)
=
\sup_{0<r<1 \land 1/|x| } \frac{1}{\gamma_{-1}(B(x,r))}
\int_{B(x,r)} |f(y,w)| \, \gamma_{-1}(y) dy.
$$
As it was shown in \cite[p. 348]{HTV2},
there exists $C>0$ such that
\begin{equation}\label{expy-x}\frac{1}{C} e^{|y|^2}
\leq
e^{|x|^2}
\leq
C e^{|y|^2},
\quad |x-y| \leq 
1 \land \frac{1}{|x|}.
\end{equation}
This property allows us to show that
\begin{equation}\label{eq:2.1}
\mathcal{M}^X_{\gamma_{-1},loc}(f)
\lesssim
\mathcal{M}^X(f),
\end{equation}
where $\mathcal{M}^X$ denotes the Hardy-Littlewood maximal function with respect to the Lebesgue measure (see \eqref{MaxHL} above), and
\begin{equation}\label{eq:2.2}
\mathcal{M}^X_{loc}(f)
\lesssim
\mathcal{M}^X_{\gamma_{-1},loc}(f).
\end{equation}

\quad

\begin{proof}[Proof of Theorem \ref{Th:1.3}]
$(a) \Rightarrow (b)$
Since $\mathcal{M}^X_{\gamma_{-1},glob}$
is bounded from $L^1_X(\Rn,\gamma_{-1})$ into itself (Remark \ref{Rem:Mglob}), $(b)$ follows {from $(a)$} by using \eqref{eq:2.1} and \cite[Theorem 1.7]{GCMT1}.
\end{proof}

\begin{proof}[Proof of Theorem \ref{Th:1.3}]
$(b) \Rightarrow (a)$
Assume that $\mathcal{M}^X_{\gamma_{-1}}$
is bounded from $L^1_X(\Rn,\gamma_{-1})$ into
$L^{1,\infty}_X(\Rn,\gamma_{-1})$.
Since $\mathcal{M}^X_{\gamma_{-1},glob}$ is bounded from $L^1_X(\Rn,\gamma_{-1})$ into itself (Remark \ref{Rem:Mglob}), $\mathcal{M}^X_{\gamma_{-1},loc}$ is bounded from $L^1_X(\Rn,\gamma_{-1})$ into
$L^{1,\infty}_X(\Rn,\gamma_{-1})$.
Now, by using \eqref{eq:2.2} we deduce that $\mathcal{M}^X_{loc}$ is bounded from
$L^1_X(\Rn,\gamma_{-1})$ into
$L^{1,\infty}_X(\Rn,\gamma_{-1})$.
According to \cite[Proposition 3.2.5]{Sa},
$\mathcal{M}^X_{loc}$ is also bounded from
$L^1_X(\Rn,dx)$ into
$L^{1,\infty}_X(\Rn,dx)$.
Furthermore, $\mathcal{M}^X$ is invariant under dilations, so proceeding as in the proof of
{\cite[$(ii) \Rightarrow (i)$, Theorem 1.10]{HTV1}} we deduce that 
$\mathcal{M}^X$ is bounded from $L^1_X(\Rn,dx)$ into
$L^{1,\infty}_X(\Rn,dx)$, or equivalently, $X$ has the Hardy-Littlewood property.
\end{proof}

\begin{proof}[Proof of Theorem \ref{Th:1.3}]
$(b) \Rightarrow (c)$
This implication is clear.
\end{proof}

\begin{proof}[Proof of Theorem \ref{Th:1.3}]
$(c) \Rightarrow (a)$
Let $f \in L^1_X(\Rn,\gamma_{-1})$ and
suppose that
$$\mathcal{M}^X_{\gamma_{-1}}(f)(x,\cdot) \in X,
\quad
\text{for almost all } 
x \in \Rn.$$
Since $\mathcal{M}^X_{\gamma_{-1},glob}$ is bounded from $L^1_X(\Rn,\gamma_{-1})$ into
itself (Remark \ref{Rem:Mglob}), we also have that
$$\mathcal{M}^X_{\gamma_{-1},glob}(f)(x,\cdot) \in X,
\quad 
\text{for almost all } x \in \Rn.$$
Hence,
$$
\mathcal{M}^X_{\gamma_{-1},loc}(f)(x,\cdot) \in X,
\quad 
\text{for almost all }
x \in \Rn$$
By \eqref{eq:2.2}, we deduce
$$
\mathcal{M}^X_{loc}(f)(x,\cdot) \in X,
\quad
\text{for almost all }
x \in \Rn.$$ 

Assume now that $f \in L^1_X(\Rn,dx)$.
Let $k \in \mathbb{N}$. We have that
$|y|\leq 1+k$, provided that $|x|\leq k$
and
$|x-y| \leq  1 \land 1/|x| $.
Then,
$$
\mathcal{M}^X_{loc}(f)(x,w)
=
\mathcal{M}^X_{loc}(f \chi_{B(0,1+k)})(x,w),
\quad w \in \Omega,
\quad |x|\leq k.
$$
Since $f \chi_{B(0,1+k)}\in L^1_X(\Rn,\gamma_{-1})$, it follows that
$$\mathcal{M}^X_{loc}(f)(x,\cdot) \in X,
\quad 
\text{for almost all } 
x \in B(0,k).$$
We conclude that
\begin{equation}\label{eq:M1}
\mathcal{M}^X_{loc}(f)(x,\cdot) \in X,
\quad 
\text{for almost all } x \in \Rn.
\end{equation}

\quad\\
On the other hand, we have that
\begin{align*}
\mathcal{M}^X_{glob}(f)(x,\cdot)
& =
\sup_{r>0}
\frac{1}{|B(x,r)|}
\int_{B(x,r)} |f(y,\cdot)|
\chi_{N^c}(x,y) \, dy \\
& =
\sup_{r>1 \land 1/|x| }
\frac{1}{|B(x,r)|}
\int_{B(x,r)} |f(y,\cdot)|
\chi_{N^c}(x,y) \, dy \\
& \lesssim
\frac{1}{(1 \land 1/|x|)^n}
\int_{\Rn} |f(y,\cdot)| \, dy,
\quad x \in \Rn.
\end{align*}
Here the integrals are understood in the $X$-B\"ochner sense.\\

Since $f \in L^1_X(\Rn,dx)$,
$$
\int_{\Rn} |f(y,\cdot)| \, dy \in X
$$
and then
\begin{equation}\label{eq:M2}
\mathcal{M}^X_{glob}(f)(x,\cdot) \in X,
\quad 
x \in \Rn.
\end{equation}

\quad \\
{Thus, by combining \eqref{eq:M1}
and \eqref{eq:M2} and using \cite[Proposition 4.12]{HTV1} we obtain $(a)$.}
\end{proof}

\section{Proof of Theorem \ref{Th:1.4}}


We prove the results for the heat semigroup $\{W_t\}_{t>0}$. For the Poisson semigroup $\{P_t\}_{t>0}$ one can proceed similarly.\\

Observe that, for every $1\le p<\infty$ and $0 \leq f\in L^p(\Rn,dx)$, one has
\begin{equation}\label{eq:equivalence}
W_{*,0}(f)
\approx
\mathcal{M}(f).
\end{equation}
Indeed, since 
$$
\frac{1}{r^n}
=\frac{r}{
\Big( \frac{r^2}{2}+\frac{r^2}{2}\Big)^{(n+1)/2}}
\lesssim \frac{r}{\left( \frac{r^2}{2}+\frac{|x-y|^2}{2}\right)^{(n+1)/2}}, \quad y\in B(x,r),
$$
then, for $1\le p<\infty$ and $0 \leq f\in L^p(\Rn,dx)$, we can write
\begin{align*}
\sup_{r>0}\frac{1}{|B(x,r)|}\int_{B(x,r)}f(y)dy 
&\lesssim 
\sup_{r>0} \int_{B(x,r)}\frac{r}{\left( \frac{r^2}{2}+\frac{|x-y|^2}{2}\right)^{(n+1)/2}}f(y)dy\\  
&\lesssim
\sup_{r>0}P_r(f)(x)
\lesssim
\sup_{r>0}W_r(f)(x), \quad x\in\Rn,
\end{align*}
where in the last inequality we have used the subordination formula (see \eqref{eq:subordination}). The converse inequality in
\eqref{eq:equivalence} follows from
\cite[Proposition 2.7]{Duo}. \\

Therefore, property $(i)$ of Theorem \ref{Th:1.4} holds (see \cite[p. 25]{HTV1}).\\

An inductive procedure allows us to see that, for every $k\in \mathbb{N}$, there exist $a_0, a_1,...,a_k\in \mathbb{R}$ such that
$$
t^{k} \partial_t^{k} W_t(z)
=
W_t(z)
\sum_{j=0}^k a_j
\Big(
\frac{|z|^2}{t}
\Big)^j,
\quad z \in \Rn, \quad t>0.
$$

Let $k\in \mathbb{N}$ and 
{$f\in L^p(\Rn,dx)$}, $1\le p<\infty$. It follows that there exist $c>0$ such that
$$
|t^{k} \partial_t^{k} W_t(z)|
\lesssim
W_{ct}(z), \,\,\,z\in \Rn,\,\,\,t>0.
$$
Then,
\begin{equation}\label{3*1}
W_{*,k}(f)
\lesssim 
W_{*,0}(f).
\end{equation}

Suppose now that $\alpha\in (k-1,k)$. We can write
$$
\partial_t^\alpha W_t(f)=\frac{1}{\Gamma(k-\alpha)}\int_0^\infty \partial_t^kW_{t+s}(f)s^{k-\alpha-1}ds,\,\,\,t>0,
$$
which implies
\begin{align}\label{3*2}
W_{*,\alpha}(f)
&\lesssim 
W_{*,k}(f) \, 
\Big( t^\alpha \, \int_0^\infty\frac{s^{k-\alpha-1}}{(s+t)^k}ds
\Big)
\lesssim
W_{*,k}(f).
\end{align}

From  (\ref{3*1}), (\ref{3*2}) and $(i)$ we deduce  $(ii)$.
\qed

\section{Proof of Theorem \ref{Th:1.1}}

The proof of Theorem \ref{Th:1.1} is presented 
in Section \ref{subsect:Th1.2}. First
we need to investigate the 
$L^p_X(\Rn, \gamma_{-1})$-boundedness 
properties for the maximal operator 
$T_{*,0}^\CA$ (Section \ref{subsect:T0})
and 
$T_{*,1}^\CA$ (Section \ref{subsect:T1}).

\subsection{$L^p$-boundedness 
properties for $T_{*,0}^\CA$}\label{subsect:T0}


The goal of this section is to establish the following result.

\begin{Prop}\label{Prop4.1}
Let $X$ be a  K\"othe function space.
\begin{itemize}
    \item[$(a)$] For $1<p<\infty$, $T_{*,0}^\CA$ is bounded from $L^p_X(\Rn, \gamma_{-1})$ into itself if, and only if, $W_{*,0}$ is bounded from $L^p_X(\Rn, dx)$ into itself.
    \item[$(b)$] $T_{*,0}^\CA$ is bounded from $L^1_X(\Rn, \gamma_{-1})$ into $L^{1,\infty}_X(\Rn, \gamma_{-1})$  if, and only if, $W_{*,0}$ is bounded from $L^1_X(\Rn, dx)$ into {$L^{1,\infty}_X(\Rn, dx)$}.
\end{itemize}
\end{Prop}

The proof of this proposition is divided into the Lemmas 
\ref{lem1}, \ref{lem2} and \ref{lem3}  below.\\

Consider the local and global maximal operators defined  by
$$
T_{*,0,loc}^{\CA} (f)(x,\omega)
:=\sup_{t>0}
\Big|T_t^{\CA} \Big(
f(\cdot, \omega)\chi_N(\cdot,x)\Big)(x)
\Big|, 
\quad x \in \Rn \text{ and } \omega \in \Omega
$$
and
$$
T_{*,0,glob}^{\CA} (f)(x,\omega)
:=\sup_{t>0}
\Big|T_t^{\CA} 
\Big(f(\cdot, \omega)\chi_{N^c}(\cdot,x)\Big)
(x)\Big|,
\quad x\in\Rn \text{ and } \omega\in\Omega.
$$
In a similar way we introduce the operators $W_{*,0,loc}$ and $W_{*,0,glob}$.

\begin{Lem}\label{lem1}
For every $1\le p<\infty$,
$T_{*,0,loc}^{\CA}-W_{*,0,loc}$ is bounded from
$L^p_X(\Rn, dx)$ into itself and from $L^p_X(\Rn, \gamma_{-1})$ into itself.
\end{Lem}

\begin{proof}
Let $1\le p<\infty$ and 
{$f\in L^p_X(\Rn, dx)$}.
We have that
\begin{align*}
& |T_{*,0,loc}^{\CA} (f)(x,\omega)-W_{*,0,loc} (f)(x,\omega)|\\
& \qquad \le \int_{\Rn} \sup_{t>0}
\Big|\Big(T_{t}^{\CA} (x,y)-W_t(x,y)\Big)\chi_N(\cdot,x)\Big| \, |f(y,\omega)|\,
dy, \quad x\in\Rn, \quad\omega\in\Omega.
\end{align*}
We can write
\begin{align*}
|T_{t}^{\CA} (x,y)-W_t(x,y))|
&\lesssim  
\frac{1-e^{-nt}}{(1-e^{-2t})^{n/2}} 
\exp\Big( \frac{-|x-e^{-t}y|^2}{1-e^{-2t}}\Big)\\
& \qquad  +
\Big|
\frac{1}{(1-e^{-2t})^{n/2}}-\frac{1}{(2t)^{n/2}}
\Big|
\exp\Big( \frac{-|x-e^{-t}y|^2}{1-e^{-2t}}\Big)
\\
& \qquad  +
\frac{1}{(2t)^{n/2}}
\Big|
\exp\Big( \frac{-|x-e^{-t}y|^2}{1-e^{-2t}}\Big)
-
\exp\Big( \frac{-|x-e^{-t}y|^2}{2t}\Big)
\Big|\\
& \qquad  +
\frac{1}{(2t)^{n/2}}
\Big|
e^{-|x-e^{-t}y|^2/2t}-e^{-|x-y|^2/2t}
\Big| \\
&=:\sum_{j=1}^4 H_j(t,x,y), 
\quad x,y\in\Rn \text{ and } t>0.
\end{align*}
Consider the change of variables $t=\log\frac{1+s}{1-s}$, that defines an increasing mapping from $(0,1)$ onto $(0,\infty).$ { By using  the fact that
\begin{equation}\label{t01}
{t}\approx{(1-e^{-at})}, \text{   for every } a>0, \text{ whenever }   t\in (0,1)
\end{equation} and \eqref{expy-x}} we get
\begin{align}\label{H1}
H_1(t,x,y)\Big|_{t=\log\frac{1+s}{1-s}}
&\lesssim
\frac{t}{(1-e^{-2t})^{n/2}}
\exp\Big( \frac{-|x-e^{-t}y|^2}{1-e^{-2t}}\Big)
\Big|_{t=\log\frac{1+s}{1-s}} \nonumber \\
& \lesssim
\frac{1}{(1-e^{-2t})^{n/2-1}}
\exp\Big( \frac{-|x-e^{-t}y|^2}{1-e^{-2t}}\Big)
\Big|_{t=\log\frac{1+s}{1-s}}\nonumber\\
& \lesssim
\frac{1}{s^{n/2-1}}
e^{-|x-y|^2/4s} 
e^{-s|x+y|^2/4}
e^{-(|x|^2-|y|^2)/2} \nonumber \\
&\lesssim
\frac{ e^{- |x-y|^2/4s}}{s^{n/2-1}}
 \lesssim
\frac{s^{1/2}}{|x-y|^{n-1}}
\lesssim
\frac{1}{|x-y|^{n-1}}, \quad  t\in (0,1), \quad (x,y)\in N. 
\end{align}
On the other hand, the Mean Value Theorem leads to
$$
\Big|
\frac{1}{(1-e^{-2t})^{n/2}}
-
\frac{1}{(2t)^{n/2}}
\Big| 
\lesssim
\frac{1}{t^{n/2-1}}, \quad  t\in (0,1).
$$
Then, by proceeding as in \eqref{H1}, we obtain that
$$
H_2(t,x,y)
\lesssim
\frac{1}{t^{n/2-1}}
\exp\Big( \frac{-|x-e^{-t}y|^2}{1-e^{-2t}}\Big)
\lesssim \frac{1}{|x-y|^{n-1}}, \quad t\in (0,1), \quad (x,y)\in N.\nonumber
$$
Now, since $e^{-2t}-1+2t>0$ for $t>0$ and \eqref{t01}, we have that
\begin{align*}
& \Big|
\exp\Big( \frac{-|x-e^{-t}y|^2}{1-e^{-2t}}\Big)
-
\exp\Big( \frac{-|x-e^{-t}y|^2}{2t}\Big)
\Big| \\
& \qquad =
\exp\Big( \frac{-|x-e^{-t}y|^2}{2t}\Big)
\Big|
\exp\Big(|x-e^{-t}y|^2 \Big(\frac{1}{2t}-\frac{1}{{1-e^{-2t}}}\Big)\Big)
-1
\Big|\\
& \qquad \lesssim
\exp\Big( \frac{-|x-e^{-t}y|^2}{2t}\Big) |x-e^{-t}y|^2\frac{e^{-2t}-1+2t}{t(1-e^{-2t})}\\
& \qquad \lesssim
\exp\Big( \frac{-|x-e^{-t}y|^2}{2t}\Big)
|x-e^{-t}y|^2, \quad x,y\in\Rn \text{ and } t\in (0,1).
\end{align*}
Thus,
\begin{align*}
H_3(t,x,y)
& \lesssim
e^{-|x-e^{-t}y|^2/2t} \frac{|x-e^{-t}y|^2}{(2t)^{n/2}} 
\lesssim
\frac{e^{-c|x-e^{-t}y|^2/t}}{t^{n/2-1}}
\lesssim
\frac{1}{|x-y|^{n-1}}, \quad t\in (0,1), \, (x,y)\in N.\nonumber
\end{align*}

Finally, observe that
    \begin{align*}
        |x-e^{-t}y|^2-|x-y|^2&=|x-y+y(1-e^{-t})|^2-|x-y|^2\\
       & =|x-y|^2+|y|^2|1-e^{-t}|^2{+}2\langle x-y,y\rangle(1-e^{-t})-|x-y|^2\\
       &=|y|^2|1-e^{-t}|^2{+}2\langle x-y,y\rangle (1-e^{-t}), \quad x,y\in\Rn, \text{ and }  t>0.
    \end{align*}
Then, by using \eqref{t01} we get that,
for $x,y\in\Rn$ and 
$0<t<1 \land 1/|x|^2$,
\begin{align*}
\Big||x-e^{-t}y|^2-|x-y|^2\Big|
& \le 
|y|^2|1-e^{-t}|^2+2| x-y| \, |y| \, (1-e^{-t}) \nonumber \\
& \lesssim
|y|^2t^{3/2}
\Big(1 \land \frac{1}{|x|^2}\Big)^{1/2}
+|x-y| \, |y| \, t.
\end{align*}
Similarly,
\begin{align*}
\Big||x-e^{-t}y|^2-|x-y|^2\Big| 
&\le
|y|^2|1-e^{-t}|^2+2| x-e^{-t}y| \, |y| \, (1-e^{-t})\nonumber\\
& \lesssim
|y|^2t^{3/2}
\Big(1 \land \frac{1}{|x|^2}\Big)^{1/2}
+|x-e^{-t}y| \, |y| \, t,
\end{align*}
for any $x,y\in\Rn$ and 
$0<t<1 \land 1/|x|^2$.    
From the Mean Value Theorem and the previous estimates it follows that
\begin{align}\label{H4}
H_4(t,x,y)
&\lesssim
\exp\Big(-
\frac{(|x-y| \land |x-e^{-t}y|)^2}{2t}\Big) \frac{||x-e^{-t}y|^2-|x-y|^2|}{{t^{n/2+1}}}\nonumber\\
&\lesssim
\frac{|y|}{t^{n/2-1/2}}
\exp\Big(-c\frac{(|x-y| \land |x-e^{-t}y|)^2}{t}\Big) \nonumber \\
&\lesssim
\frac{1+|x|}{|x-y|^{n-1}}, 
\quad (x,y)\in N, \quad
0<t<1 \land \frac{1}{|x|^2}.
\end{align}
Therefore, we conclude that
$$
\sup_{0<t<1 \land 1/|x|^2}  |T_{t}^{\CA} (x,y)-W_t(x,y))|
\lesssim
\frac{{1+|x|}}{|x-y|^{n-1}}, \quad  (x,y)\in N.
$$
On the other hand, we have that
\begin{align*}
\sup_{t\ge1 \land 1/|x|^2} 
|T_{t}^{\CA} (x,y)-W_t(x,y)|
&\lesssim
\Big( 1 \land \frac{1}{|x|} \Big)^{-n} 
\lesssim
\frac{1+|x|}{|x-y|^{n-1}}, 
\quad (x,y)\in N.\nonumber
\end{align*}
Hence,
\begin{align}\label{TWle1}
\sup_{t>0}  |T_{t}^{\CA} (x,y)-W_t(x,y))| 
\lesssim
\frac{1+|x|}{|x-y|^{n-1}}, \quad   (x,y)\in N.
\end{align}

By taking into account that 
$$1 \land \frac{1}{|x|} 
\approx \frac{1}{1+|x|}
\approx \frac{1}{1+|y|}, \quad  (x,y)\in N,$$
we get
    \begin{equation}\label{Lg}
     \sup_{x\in\Rn}\int_{\Rn}\frac{1+|x|}{|x-y|^{n-1}}\chi_N(x,y)dy+\sup_{y\in\Rn}\int_{\Rn}\frac{1+|x|}{|x-y|^{n-1}}\chi_N(x,y)dx<\infty.
     \end{equation}
 Then, Schur's lemma guarantees that the operator $L$ defined by
\begin{equation}\label{eq:operatorL}
L(g)(x)
:=
\int_{\Rn}\frac{1+|x|}{|x-y|^{n-1}}g(y)\chi_N(x,y)dy
\end{equation}
is bounded from $L^p(\Rn,dx)$ into itself, for every $1\le p <\infty.$\\

We have obtained that for every $1\le p<\infty$, $T_{*,0,loc}^{\CA}-W_{*,0,loc}$ is bounded from $L^p_X(\Rn, dx)$ into itself. Moreover, according to
\cite[Proposition 3.2.5]{Sa} (which is indeed true for $p=1$), it is also bounded from $L^p_X(\Rn, \gamma_{-1})$ into itself.
  \end{proof}

\begin{Lem} \label{lem2}
$  T^{\CA}_{*,0,glob}$ is bounded from $L^1_X(\Rn, \gamma_{-1})$ into $L^{1,\infty}_X(\Rn, \gamma_{-1})$.
 \end{Lem}
\begin{proof}
We can write,
for every 
$x,y\in\Rn$ and $s\in (0,1)$,
$$
T_{\log\frac{1+s}{1-s}}^\CA(x,y)
=\frac{(1-s)^n}{\pi^{n/2}(4s)^{n/2}}
\exp{\left(-\frac{|x(1-s)+y(1+s)|^2}{4s}\right)}e^{-|x|^2+|y|^2}.
$$
According to \cite[Lemma 3.3.3]{Sa} we obtain, for $(x,y)\in N^c$ and $x,y\neq 0$,
\begin{align}\label{eq: 4.01}
& \sup_{t>0} \frac{e^{-nt}}{\pi^{n/2}(1-e^{-2t})^{n/2}} 
\exp\Big(-\frac{|x-e^{-t}y|^2}{1-e^{-2t}}\Big)
\nonumber \\
& \qquad =
\sup_{s\in (0,1)}\frac{(1-s)^n}{\pi^{n/2}(4s)^{n/2}}\exp{\left(-\frac{|x(1-s)+y(1+s)|^2}{4s}\right)}e^{-|x|^2+|y|^2}\nonumber\\
& \qquad \lesssim
e^{-|x|^2+|y|^2}\Big[
(1+|x|)^n\land(|x|\sin\theta(x,y))^{-n}
\Big],
\end{align}
where $\theta(x,y)$ represents the angle between $x$ and $y\in\Rn\setminus\{0\}$, when $n>1$; and $\theta(x,y)=0$, $x,y\in\mathbb{R}\setminus\{0\}$.\\

Then, for every
$x\in\Rn$ and $\omega\in\Omega$,
\begin{align*}
T^{\CA}_{*,0,glob}(f)(x,\omega)
& \lesssim
\int_{\Rn} \sup_{t>0} \frac{e^{-nt} }{(1-e^{-2t})^{n/2}}
\exp\Big(-\frac{|x-e^{-t}y|^2}{1-e^{-2t}}\Big)
|f(y,\omega)|\chi_{N^c}(x,y) dy\\
&\lesssim
\int_{\Rn}e^{-|x|^2+|y|^2}\Big[
(1+|x|)^n\land(|x|\sin\theta(x,y))^{-n}\Big]|f(y,\omega)|\chi_{N^c}(x,y) dy.
\end{align*}
From  \cite[Lemma 3.3.4]{Sa}, we deduce that $  T^{\CA}_{*,0,glob}$ is bounded from $L^1_X(\Rn, \gamma_{-1})$ into $L^{1,\infty}_X(\Rn, \gamma_{-1})$.
\end{proof}

\begin{Lem}\label{lem3}  
For every $1<p<\infty$,
$T^{\CA}_{*,0,glob}$  is bounded from $L^p_X(\Rn,\gamma_{-1})$ into itself.
\end{Lem}

\begin{proof}
Let $1<{p}<\infty.$
Observe that
\begin{align*}
&\Big( \int_{\Rn}\sup_{t>0} \Big|T^{\CA}_{t}
\Big(f(\cdot)\chi_{N^c}(\cdot,x)
\Big)(x)\Big|^p 
e^{|x|^2}dx\Big)^{1/p}
\\
&\qquad \le 
\Big( \int_{\Rn} \Big( \int_{\Rn}\sup_{t>0} |T_t^\CA(x,y)|\chi_{N^c}(x,y)
e^{(|x|^2-|y|^2)/p}
|f(y)|e^{|y|^2/p}dy\Big)^pdx\Big)^{1/p}.
\end{align*}
Therefore, it is enough to prove that the operator associated to the kernel
$$
\sup_{t>0} |T_t^\CA(x,y)|\chi_{N^c}(x,y)e^{(|x|^2-|y|^2)/p}
$$ 
is of strong type $p$ with respect to the Lebesgue measure.\\

According to \cite[Proposition 2.1]{MPS}, we obtain, for every $(x,y)\in N^c, $
\begin{align}\label{Ttglob}
\sup_{t>0}|T_t^{\CA}(x,y)|
&\lesssim
\begin{cases}
e^{-|x|^2},  & \mbox{if  $\langle x,y\rangle \le 0$},\\
\left(\frac{|x+y|}{|x-y|}\right)^{n/2}\exp{\left( \frac{|y|^2-|x|^2}{2}-\frac{|x-y||x+y|}{2}\right)}, &\mbox{if $\langle x,y\rangle>0$},
\end{cases}\nonumber\\
&\lesssim
\begin{cases}
e^{-|x|^2},  & \mbox{if  $\langle x,y\rangle \le 0$},\\
{|x+y|}^{n}\exp{\left( \frac{|y|^2-|x|^2}{2}-\frac{|x-y||x+y|}{2}\right)}, &\mbox{if $\langle x,y\rangle>0$}.
\end{cases}
\end{align}
In the last inequality we have used that $|x+y||x-y|\gtrsim 1$ for $(x,y)\in N^c$ and $\langle x,y\rangle>0$.
Since 
$$||y|^2-|x|^2|\le |x+y||x-y|,
\quad x,y\in\Rn,$$
we get
\begin{align*}
& \int_{\Rn}
e^{(|x|^2-|y|^2)/p}
\sup_{t>0}|T_t^{\CA}(x,y)|
\chi_{N^c}(x,y)dy \\
& \qquad \le 
\int_{\langle x,y\rangle \le 0}e^{-|x|^2(1-1/p)-|y|^2/p}dy \\
& \qquad \qquad +
\int_{\langle x,y\rangle > 0}|x+y|^n \exp{\Big(-\Big(\frac{1}{2}-\Big|\frac{1}{{p}}-\frac{1}{2}\Big|\Big)|x-y||x+y|\Big)}dy.
\end{align*}
Thus, by proceeding as in \cite[p. 501]{Pe}, we obtain that
$$
\sup_{x\in\Rn}\int_{\Rn}
e^{(|x|^2-|y|^2)/p}
\sup_{t>0}|T_t^{\CA}(x,y)|\chi_{N^c}(x,y)dy<\infty.
$$
Also, we have that
$$
\sup_{y\in\Rn}\int_{\Rn}
e^{(|x|^2-|y|^2)/p}
\sup_{t>0}|T_t^{\CA}(x,y)|\chi_{N^c}(x,y)dx<\infty.
$$
We conclude that the operator $\mathbb{L}$ defined by
$$
\mathbb{L}(g)(x)
:=
\int_{\Rn}
e^{(|x|^2-|y|^2)/p}
\sup_{t>0}|T_t^{\CA}(x,y)|\chi_{N^c}(x,y)g(y)dy, \quad  x\in\Rn,
$$
is bounded from $L^p(\Rn, {dx})$ into itself. Hence, the operator $T^{\CA}_{*,0,glob}$  is bounded from $L^p_X(\Rn,\gamma_{-1})$ into itself.
\end{proof}

\begin{proof}[Proof of Proposition \ref{Prop4.1}]
Suppose that $T^{\CA}_{*,0}$ is bounded from $L^1_X(\Rn,\gamma_{-1})$  into  $L^{1,\infty}_X($ $\Rn,\gamma_{-1})$. From Lemma \ref{lem2} we know that $T^{\CA}_{*,0,glob}$ is bounded from $L^1_X(\Rn,\gamma_{-1})$  into $L^{1,\infty}_X(\Rn,\gamma_{-1})$, then the same boundeness property holds for $T^{\CA}_{*,0,loc}$. Moreover, Lemma \ref{lem1} states that
$T_{*,0,loc}^{\CA}-W_{*,0,loc}$ is bounded from $L^1_X(\Rn, \gamma_{-1})$ into $L^{1,\infty}_X(\Rn,\gamma_{-1})$, so $W_{*,0,loc}$ has also this property. Then, according to \cite[Proposition 3.2.5]{Sa}, $W_{*,0,loc}$  is bounded  from $L^1_X(\Rn, dx)$ into $L^{1,\infty}_X(\Rn,dx)$. Furthermore, since $W_{*,0}$ is dilation invariant, by proceeding as in the proof of \cite[Theorem 1.10]{HTV1}, we obtain that $W_{*,0}$ is bounded from $L^1_X(\Rn, dx)$ into $L^{1,\infty}_X(\Rn,dx)$.\\

Assume now that  $W_{*,0}$ is bounded from $L^1_X(\Rn, dx)$ into $L^{1,\infty}_X(\Rn,dx)$. We have that
\begin{equation}\label{eq:4.1}
\sup_{t>0} W_t(z)
\lesssim
\frac{1}{|z|^n}, \quad z\in\Rn\setminus\{0\}.
\end{equation}
Then, according to \cite[Propositions 3.2.5 and 3.2.7]{Sa} we deduce that $W_{*,0, loc}$ is bounded from  $L^1_X(\Rn, \gamma_{-1})$ into $L^{1,\infty}_X(\Rn,\gamma_{-1})$. Note that the size condition \eqref{eq:4.1} is sufficient to obtain this property. Furthermore, from Lemmas \ref{lem1} and \ref{lem2} we know that $T_{*,0,loc}^{\CA}-W_{*,0,loc}$ and  $T_{*,0,glob}^{\CA}$ are bounded from $L^1_X(\Rn, \gamma_{-1})$ into $L^{1,\infty}_X(\Rn,\gamma_{-1})$, we conclude that $T_{*,0}^{\CA}$ is bounded from $L^1_X(\Rn, \gamma_{-1})$ into $L^{1,\infty}_X(\Rn,\gamma_{-1})$.\\

Thus, $(b)$ is justified. Property $(a)$ can be proven in a similar way.
\end{proof}

\subsection{$L^p$-boundedness 
properties for $T_{*,1}^\CA$}
\label{subsect:T1}


Now we concentrate on the following.

\begin{Prop}\label{Prop4.5}
Let $X$ be a K\"othe function space.
\begin{itemize}
    \item[$(a)$] For $1<p<\infty$, $ T_{*,1}^{\CA}$ is bounded from $L^p_X(\Rn, \gamma_{-1})$ into itself if, and only if, $W_{*,1}$ is bounded from $L^p_X(\Rn, dx)$ into itself.
     \item[$(b)$] $T_{*,1}^{\CA}$ is bounded from $L^1_X(\Rn, \gamma_{-1})$ into $L^{1,\infty}_X(\Rn, \gamma_{-1})$ if, and only if, $W_{*,1}$ is bounded from $L^1_X(\Rn, dx)$ into $L^{1,\infty}_X(\Rn, dx)$. \end{itemize}
\end{Prop}

As in the previous section, Proposition \ref{Prop4.5} is decomposed into Lemmas 
\ref{Lem4.6},
\ref{Lem:4.8} and 
\ref{Lem:4.9}  that we present next.\\

We define
$$
T_{*,1,loc}^{\CA} (f)(x,\omega)
:=
\sup_{t>0}
\Big|t\partial_tT_t^{\CA} 
\Big(f(\cdot, \omega)\chi_N(\cdot,x)\Big)
(x)\Big|,
\quad x\in\Rn \text{ and } \omega\in\Omega
$$
and
$$
T_{*,1,glob}^{\CA} (f)(x,\omega)
:=
\sup_{t>0}
\Big|t\partial_tT_t^{\CA}
\Big(f(\cdot, \omega)\chi_{N^c}(\cdot,x)\Big)(x)
\Big|,
\quad x\in\Rn \text{ and } \omega\in\Omega.
$$
In a similar way we introduce the operators $W_{*,1,loc}$ and $W_{*,1,glob}$.

\begin{Lem}\label{Lem4.6}
 For every $1\le p<\infty$, $T_{*,1,loc}^{\CA}-W_{*,1,loc}$ is bounded from $L^p_X(\Rn, dx)$ into itself and from $L^p_X(\Rn, \gamma_{-1})$ into itself.
 \end{Lem}

\begin{proof}
Observe that, for all 
$x,y\in\Rn$ and $t>0$,
\begin{align*}
\partial_t T_t^\CA(x,y)
& =   
\frac{1}{\pi^{n/2}}  
\exp\Big( \frac{-|x-e^{-t}y|^2}{1-e^{-2t}} \Big)
\Big[
- \frac{n e^{-nt}}{(1-e^{-2t})^{(n+2)/2}} \nonumber \\
& \qquad \qquad \qquad \qquad  \qquad \quad  
-  \frac{2e^{-(n+1)t}}{(1-e^{-2t})^{(n+2)/2}} \,
\sum_{i=1}^n  y_i (x_i - e^{-t}y_i) \,
 \nonumber \\
& \qquad \qquad \qquad \qquad  \qquad \quad  
+
\frac{2e^{-(n+2)t}}{(1-e^{-2t})^{(n+4)/2}} \,
|x-e^{-t}y|^2 \,
\Big]
 \end{align*}
 and
 $$
 \partial_t W_t(x-y)=\frac{1}{(2\pi)^{n/2}}\left(-\frac{n}{2}\frac{1}{t^{n/2+1}}+\frac{|x-y|^2}{2t^{n/2+2}} \right) e^{-|x-y|^2/2t}, \quad  x,y\in\Rn, \:  t>0.
 $$

 In the local region, i.e., when $(x,y)\in N$, we are going to estimate
$$
\sup_{t>0}
\Big|t\Big( \partial_t T_t^\CA(x,y)-\partial_t W_t(x-y)\Big)\Big|.
 $$
 We can write
\begin{align*}
& t \Big(\partial_t T_t^\CA(x,y)-\partial_t W_t(x-y)\Big) \\
& \qquad = 
\frac{1}{\pi^{n/2}} 
\Big[
-
\frac{nte^{-nt}}{(1-e^{-2t})^{(n+2)/2}}
\exp\Big( \frac{-|x-e^{-t}y|^2}{1-e^{-2t}} \Big)
+\frac{nt}{(2t)^{n/2+1}}
e^{-|x-y|^2/2t}\\
& \qquad \qquad \quad +  
\frac{2te^{-(n+2)t}|x-e^{-t}y|^2 }{(1-e^{-2t})^{(n+4)/2}} \,
\exp\Big( \frac{-|x-e^{-t}y|^2}{1-e^{-2t}} \Big)-\frac{2t|x-y|^2}{(2t)^{n/2+2}} e^{-|x-y|^2/2t}
\\
& \qquad \qquad \quad
 -  \,
\frac{2te^{-(n+1)t}}{(1-e^{-2t})^{(n+2)/2}} \,
\sum_{i=1}^n  y_i (x_i - e^{-t}y_i)\exp\Big( \frac{-|x-e^{-t}y|^2}{1-e^{-2t}} \Big) \Big]\\
& \qquad =: \frac{1}{\pi^{n/2}}  \sum_{j=1}^3 H_j(t,x,y),\quad x,y\in\Rn, \: t>0.
 \end{align*}
 We decompose $H_1$ as follows
 \begin{align*}
 H_1(t,x,y)
 &= -
\frac{nt(e^{-nt}-1)}{(1-e^{-2t})^{(n+2)/2}}
\exp\Big( \frac{-|x-e^{-t}y|^2}{1-e^{-2t}} \Big)\\
&\quad -
nt
\Big[
\frac{1}{(1-e^{-2t})^{(n+2)/2}}-\frac{1}{(2t)^{n/2+1}}\Big]
\exp\Big( \frac{-|x-e^{-t}y|^2}{1-e^{-2t}} \Big)\\
&\quad -
\frac{nt}{(2t)^{n/2+1}}
\Big[
\exp\Big( \frac{-|x-e^{-t}y|^2}{1-e^{-2t}} \Big)
-
\exp\Big( \frac{-|x-e^{-t}y|^2}{2t} \Big)
\Big]\\
&\quad -
\frac{nt}{(2t)^{n/2+1}}
\Big[
\exp\Big( \frac{-|x-e^{-t}y|^2}{2t} \Big)
-
\exp\Big(-\frac{|x-y|^2}{2t}\Big) \Big]\\
&=:\sum_{j=1}^4 H_{1j}(t,x,y), \quad x,y\in\Rn, \: t>0.
 \end{align*}
 Now we will use some of the manipulations and estimates from the proof of Lemma \ref{lem1}.
 We have that
$$
|H_{11}(t,x,y)|
\lesssim
\frac{1}{t^{n/2-1}}
\exp\Big( \frac{-|x-e^{-t}y|^2}{1-e^{-2t}} \Big)
\lesssim
\frac{1}{|x-y|^{n-1}}, \quad t\in (0,1), \: (x,y)\in N.
 $$
Since 
 $$
\Big|
\frac{1}{(1-e^{-2t})^{n/2+1}}
-
\frac{1}{(2t)^{n/2+1}}
\Big|
\lesssim
\frac{ 1}{t^{n/2+2}}|2t-1+e^{-2t}|
\lesssim
\frac{1}{t^{n/2}}, \quad  t\in (0,1),
$$
it follows that
$$
|H_{12}(t,x,y)|
\lesssim 
\frac{1}{t^{n/2-1}}
\exp\Big(-\frac{|x-e^{-t}y|^2}{1-e^{-2t}}\Big)
\lesssim
\frac{1}{|x-y|^{n-1}}, \quad  t\in (0,1), \quad (x,y)\in N.\nonumber
$$
 Also, we get
 \begin{align*}
|H_{13}(t,x,y)|
&\lesssim
e^{-|x-e^{-t}y|^2/2t} \frac{|x-e^{-t}y|^2}{t^{n/2}}
\lesssim
\frac{1}{|x-y|^{n-1}}, \quad  t\in (0,1), \quad (x,y)\in N.\nonumber
\end{align*}
 Finally, we obtain
\begin{align*}
| H_{14}(t,x,y)|
&\lesssim
\frac{{1+|x|}}{|x-y|^{n-1}}, 
\quad  0<t<1 \land \frac{1}{|x|^2}, \quad (x,y)\in N.
\end{align*}

 We conclude that
$$
|H_{1}(t,x,y)|
\lesssim 
\frac{{1+|x|}}{|x-y|^{n-1}}, \quad  0<t<1 \land \frac{1}{|x|^2}, \quad (x,y)\in N.\nonumber
 $$
 On the other hand, we can write
\begin{align*}
H_2(t,x,y)
&=
\frac{2t(e^{-(n+2)t}-1)|x-e^{-t}y|^2}{(1-e^{-2t})^{n/2+2}}
\exp\Big( \frac{-|x-e^{-t}y|^2}{1-e^{-2t}} \Big)\\
&\quad +
2t\Big[
\frac{1}{(1-e^{-2t})^{n/2+2}}-\frac{1}{(2t)^{n/2+2}}
\Big] 
|x-e^{-t}y|^2\exp\Big( \frac{-|x-e^{-t}y|^2}{1-e^{-2t}} \Big)\\
&\quad +
\frac{1}{(2t)^{n/2+1}}
[|x-e^{-t}y|^2-|x-y|^2]
\exp\Big( \frac{-|x-e^{-t}y|^2}{1-e^{-2t}} \Big)\\
&\quad +
\frac{1}{(2t)^{n/2+1}}|x-y|^2
\Big[
\exp\Big( \frac{-|x-e^{-t}y|^2}{1-e^{-2t}} \Big)
-
\exp\Big(\frac{-|x-e^{-t}y|^2}{2t}\Big) \Big]\\
&\quad +
\frac{|x-y|^2}{(2t)^{n/2+1}}
\Big[
\exp\Big( \frac{-|x-e^{-t}y|^2}{2t} \Big)
-
\exp\Big(-\frac{|x-y|^2}{2t}\Big)
\Big]\\
&=:
\sum_{j=1}^5 H_{2j}(t,x,y), \quad x,y\in\Rn, \: t>0.
 \end{align*}
 We have that
$$ |H_{21}(t,x,y)|+|H_{22}(t,x,y)|
\lesssim
\frac{1}{|x-y|^{n-1}}, \quad  t\in (0,1), \quad (x,y)\in N.\nonumber
 $$
 Moreover, by proceeding as in \eqref{H4} we get
 \begin{align*}
|H_{23}(t,x,y)|     
& \lesssim 
\frac{|y|}{t^{n/2-1/2}}\exp{\Big(-c\frac{(|x-y| \land |x-e^{-t}y|)^2}{t}\Big)}\\
&\lesssim
\frac{1+|x|}{|x-y|^{n-1}}, 
\quad  (x,y)\in N, 
\quad  
0<t<1 \land \frac{1}{|x|^2}.
 \end{align*}
 For $H_{24}$ we obtain
\begin{align*}
|H_{24}(t,x,y)|
&\lesssim
\frac{|x-y|^2}{t^{n/2+1}} \exp{\Big(-\frac{|x-e^{-t}y|^2}{{2t}}\Big)} {|x-e^{-t}y|^2}
        \\
&\lesssim
\frac{1}{|x-y|^{n-1}}, \quad  t\in (0,1), \quad (x,y)\in N.\nonumber
\end{align*}
Finally, for $H_{25}$ we get
\begin{align*}
|H_{25}(t,x,y)|
&\lesssim
\frac{|x-y|^2}{{t^{n/2+1}}} \exp{\Big(-c\frac{
( |x-y| \land |x-e^{-t}y|)^2
}{t}\Big)}{ \frac{||x-e^{-t}y|^2-|x-y|^2|}{t}} \\
& \lesssim
{\frac{|x-y|^2}{t}\frac{|y|}{t^{n/2-1/2}}\exp{\Big(-c\frac{( |x-y| \land |x-e^{-t}y|)^2}{t}\Big)}}\\
&{\lesssim  \frac{1+|x|}{|x-y|^{n-1}}, \quad  (x,y)\in N, \quad 
0<t<1 \land \frac{1}{|x|^2}.}
\end{align*}
 We conclude that
 $$
 |H_{2}(t,x,y)|
 \lesssim
 { \frac{1+|x|}{|x-y|^{n-1}}, \quad  (x,y)\in N, 
 \quad 0<t<1 \land \frac{1}{|x|^2}.}
 $$

 Finally,
 \begin{align*}
|H_{3}(t,x,y)|  
& \lesssim
\frac{|y|}{t^{n/2-1/2}}
\exp{\Big(-c\frac{ |x-e^{-t}y|^2}{t}\Big)}\\
&\lesssim
\frac{1+|x|}{|x-y|^{n-1}}, \quad  (x,y)\in N, \quad 0<t< 1.\nonumber
 \end{align*}
 
Summarizing, we have obtained that
$$
\sup_{0<t<1 \land 1/|x|^2}  
\Big|t \Big(
\partial_t T_t^\CA(x,y)
-
\partial_t W_t(x-y) \Big) \Big|
\lesssim
\frac{{1+|x|}}{|x-y|^{n-1}}, \quad   (x,y)\in N.
    $$
    
On the other hand, by using that 
$$
1-e^{-2t}
\ge 
1-e^{-2 (1 \land 1/|x|)},
\quad t\ge1 \land \frac{1}{|x|^2}$$
and \eqref{t01}, we get
\begin{align*}
\sup_{t\ge1 \land 1/|x|^2}  
\Big|t\Big( 
\partial_t T_t^\CA(x,y)
-
\partial_t W_t(x-y)\Big)\Big|
&\lesssim
\frac{1}{
(1  \land 1/|x| )^n
}+\frac{|y|}{
(1  \land 1/|x| )^{n-1}
}  \\
& \lesssim
\frac{1+|x|}{|x-y|^{n-1}}, \quad   (x,y)\in N.
\end{align*}

Therefore,
\begin{align*}
\sup_{t>0}  
\Big|t \Big(
\partial_t T_t^\CA(x,y)
-
\partial_t W_t(x-y)\Big) \Big|
\lesssim
\frac{1+|x|}{|x-y|^{n-1}}, \quad   (x,y)\in N.
\end{align*}
Recall that the operator $L$ 
given by \eqref{eq:operatorL}
is bounded from $L^p(\Rn,dx)$ into itself, for every $1\le p <\infty$, see \eqref{Lg}.  
Thus, for every $1\le p<\infty$, $T_{*,1,loc}^{\CA}-W_{*,1,loc}$ is bounded from $L^p_X(\Rn, dx)$ into itself.
Furthermore, 
\cite[Proposition 3.2.5]{Sa} (which is indeed true for $p=1$) also implies the boundedness from  $L^p_X(\Rn, \gamma_{-1})$ into itself.
\end{proof}
\begin{Rem}
Observe that Lemma \ref{Lem4.6} also holds if in the definition of the local operators we replace $N$
 by $N_{\beta}$, for any $\beta>0.$
 \end{Rem}

\begin{Lem}\label{Lem:4.8}
 $T_{*,1,glob}^{\CA}$ is bounded from $L^1_X(\Rn, \gamma_{-1})$ into  $L^{1,\infty}_X(\Rn, \gamma_{-1})$.
\end{Lem}

\begin{proof}
We can write
\begin{align}\label{eq:partialtHt}
t\partial_tT_t^\CA(x,y)
& =  - \frac{1}{\pi^{n/2}} \,
\frac{nt e^{-nt}}{(1-e^{-2t})^{(n+2)/2}} \,
\exp\Big( \frac{-|y-e^{-t}x|^2}{1-e^{-2t}} \Big)  \, e^{|y|^2-|x|^2} \nonumber \\
& \quad - \frac{2}{\pi^{n/2}} \,
\frac{te^{-(n+1)t}}{(1-e^{-2t})^{(n+2)/2}} \,
\sum_{i=1}^n  x_i (y_i - e^{-t}x_i) \,
 \exp\Big( \frac{-|y-e^{-t}x|^2}{1-e^{-2t}} \Big) \,
 e^{|y|^2-|x|^2} \nonumber \\
& \quad + \frac{2}{\pi^{n/2}} \,
\frac{te^{-(n+2)t}}{(1-e^{-2t})^{(n+4)/2}} \,
|y-e^{-t}x|^2 \,
 \exp\Big( \frac{-|y-e^{-t}x|^2}{1-e^{-2t}} \Big) \,
 e^{|y|^2-|x|^2} \nonumber \\
& = : \sum_{j=1}^3K_j(t,x,y), \quad x,y\in\Rn, \quad t>0.
\end{align}
We introduce the following global maximal operators
\begin{equation}\label{eq:Kjstar}
K^*_{j,glob}(g)(x)
:=
\int_{\Rn}   \sup_{t>0}  |K_j(t,x,y)| \, \chi_{N^c}(x,y) \, g(y) \, dy ,\quad x\in\Rn,
\quad j=1,2,3.
\end{equation}
We are going to study these operators separately by adapting the ideas developed in the proof of \cite[Proposition 5.1]{BrSj}.\\


\textbf{Step 1: $K^*_{1,glob}$ is bounded from $L^1(\Rn, \gamma_{-1})$ into $L^{1,\infty}(\Rn, \gamma_{-1})$.}\\

Observe that, by taking $r=e^{-t}$, $r\in (0,1),$
\begin{align}\label{eq:K1r}
|K_1(t,x,y)|_{r=e^{-t}}|
& =\frac{n}{\pi^{n/2}} \frac{(-\log r) r^n}{(1-r^2)^{(n+2)/2}} \,
\exp\Big( \frac{-|y-r x|^2}{1-r^2} \Big)  \, e^{|y|^2-|x|^2}, \quad x,y\in\Rn.
\end{align}
To continue the analysis it is convenient to consider the cases $0<r<1/2$ and $1/2 \leq r<1$, or equivalently $\log2<t<\infty$ and $0<t \leq \log2$. Our objective is to estimate
$$
\sup_{0<t\le \log 2}|K_1(t,x,y)|, \quad (x,y)\in N^c,
$$
and
$$\sup_{t
>\log 2}|K_1(t,x,y)|, \quad (x,y)\in N^c.
$$
Assume first that $1/2 \leq r<1$.
Since \eqref{t01} holds, we deduce  the pointwise estimate (see \eqref{eq: 4.01})
\begin{align}\label{maxtLog}
\sup_{0<t\le \log 2}|K_1(t,x,y)|
&\lesssim 
\sup_{0<t\le \log 2}|T_t^\CA(x,y)|\nonumber\\
&\lesssim
e^{|y|^2-|x|^2}\Big[
(1+|x|)^n\land(|x|\sin\theta(x,y))^{-n}
\Big], \quad (x,y)\in N^c, \quad  x,y\neq 0.\,\
\end{align}

Consider now the case of  $0<r<1/2$. Then, \eqref{eq:K1r} can be controlled by
\begin{align*}
|K_1(t,x,y)|_{r=e^{-t}}|
& \lesssim
(-\log r) r^n
e^{ - c|y-r x|^2}  \, e^{|y|^2-|x|^2}
\lesssim
r^{n-1} e^{ - c|y-r x|^2}  \, e^{|y|^2-|x|^2}, \quad x,y\in\Rn.
\end{align*}
Moreover, if $|y| >2|x|$, it follows that
\begin{align*}
\sup_{t>\log 2}|K_1(t,x,y)|
& \lesssim
e^{ -c |y|^2}  \, e^{|y|^2-|x|^2}
\lesssim
\frac{e^{|y|^2-|x|^2}}{|y|^{n-1}}
\lesssim
\frac{e^{|y|^2-|x|^2}}{|x|^{n-1}},\quad (x,y)\in N^c.
\end{align*}

On the other hand,
 by using that $|y-rx|^2=|y_{\perp}|^2+|r-r_0|^2|x|^2$, $x,y\in\Rn,$ we get, for  $|y| \le  2|x|$,
\begin{align*}
|K_1(t,x,y)|_{r=e^{-t}}|
& \lesssim
(|r_0|^{n-1} + |r-r_0|^{n-1})
e^{ -|y_\perp|^2} \,
e^{ -|r-r_0|^2|x|^2} \,
e^{|y|^2-|x|^2} \nonumber \\
& \lesssim
e^{ -|y_\perp|^2}
\Big[\Big(\frac{|y|}{|x|}\Big)^{n-1} e^{ -|r-r_0|^2|x|^2}
+ |r-r_0|^{n-1} e^{ -|r-r_0|^2|x|^2} \Big]
e^{|y|^2-|x|^2} \nonumber \\
& \lesssim
\Big[e^{ -|y_\perp|^2}\Big(\frac{|y|}{|x|}\Big)^{n-1} |x|
+ |x|^{1-n} \Big]
e^{|y|^2-|x|^2}, \quad (x,y) \in N^c,
\end{align*}
where $
r_0:= \frac{|y|}{|x|} \cos \theta
$,  $\theta$ is the angle between $x$ and $y$, and
$
y=:y_x + y_\perp,
$ \text{with}  $y_x \parallel x$
 \text{and} $\quad y_\perp \perp x.  $
 In the last inequality we have taken into account that
$|x| \geq C$, provided that $|x|\ge |y|/2$ and $(x,y) \in N^c$.\\

Therefore,
\begin{align}\label{maxtLog2}
\sup_{t> \log 2}|K_1(t,x,y)|
&\lesssim
e^{|y|^2-|x|^2}\Big[|x|^{1-n}+e^{ -|y_\perp|^2}\Big(\frac{|y|}{|x|}\Big)^{n-1} |x|\,\chi_{\{|y|\le 2|x|\}}
 \Big], \quad (x,y)\in N^c.\,\
\end{align}
By combining \eqref{maxtLog} and \eqref{maxtLog2} we get,
for $(x,y)\in N^c$,
\begin{align*}
\sup_{t>0}|K_1(t,x,y)|
&\lesssim e^{|y|^2-|x|^2}\Big[|x|^{1-n}+e^{ -|y_\perp|^2}\Big(\frac{|y|}{|x|}\Big)^{n-1} |x|\,\chi_{\{|y|\le 2|x|\}}
 \\
 & \qquad \qquad \quad +
(1+|x|)^n\land(|x|\sin\theta(x,y))^{-n}
\Big].
\end{align*}
Then, according to \cite[Lemma 3.3.4.]{Sa} and \cite[Lemmas 4.2 and 4.3]{BrSj} we deduce that  the operator
$K_{1,glob}^*$ is bounded from $L^1(\Rn, \gamma_{-1})$ into $L^{1,\infty}(\Rn,\gamma_{-1})$.\\


\textbf{Step 2: $K_{2,glob}^*$ is bounded from $L^1(\Rn, \gamma_{-1})$ into $L^{1,\infty}(\Rn, \gamma_{-1})$.}\\

Denoting again $r=e^{-t},$ we have
\begin{align}\label{eq:Kt2}
|K_2(t,x,y)|_{r=e^{-t}}|
& \lesssim
\frac{(-\log r) r^{n+1}}{(1-r^2)^{(n+2)/2}}\,
\sum_{i=1}^n | x_i (y_i - r x_i) | \,
\exp\Big( \frac{-|y-r x|^2}{1-r^2} \Big)  \, e^{|y|^2-|x|^2} \nonumber \\
& \lesssim
\frac{(-\log r) r^{n+1}}{(1-r^2)^{(n+1)/2}}\,
 |x| \Big(\frac{|y - r x|^2}{1-r^2}\Big)^{1/2} \,
\exp\Big( \frac{-|y-r x|^2}{1-r^2} \Big)  \, e^{|y|^2-|x|^2} \nonumber \\
& \lesssim
\frac{ r^{n}}{(1-r^2)^{(n+1)/2}}
\, |x| \,
\exp\Big( -\frac{1}{2}\frac{|y-r x|^2}{1-r^2} \Big)  \, e^{|y|^2-|x|^2},
\quad x,y\in\Rn, \: t>0.
\end{align}
As in the previous case,  we want to estimate
$$
\sup_{0<t\le \log 2}|K_2(t,x,y)|, \quad (x,y)\in N^c
$$
and
$$\sup_{t
>\log 2}|K_2(t,x,y)|, \quad (x,y)\in N^c.
$$

Let $0<r<1/2$. If $|y| > 2|x|$, we can control the kernel in \eqref{eq:Kt2} by
\begin{align*}
|K_2(t,x,y)|
& \lesssim
|y|
e^{-c|y|^2}  \, e^{|y|^2-|x|^2}
\lesssim
\frac{e^{|y|^2-|x|^2}}{|y|^{n-1}}
\lesssim
\frac{e^{|y|^2-|x|^2}}{|x|^{n-1}}, 
\quad x,y\in\Rn.
\end{align*}
Furthermore, when $|y| \le 2|x|$, then $|r_0|\le 2$ and we get
\begin{align*}
|K_2(t,x,y)|
& \lesssim
(|r_0|^{n} + |r-r_0|^{n}) |x|
e^{ -|y_\perp|^2} \,
e^{ -|r-r_0|^2|x|^2} \,
e^{|y|^2-|x|^2} \nonumber\\
& \lesssim
\Big[e^{ -|y_\perp|^2}\Big(\frac{|y|}{|x|}\Big)^{n-1} |x|
+ |x|^{1-n} \Big]
e^{|y|^2-|x|^2}, \quad (x,y) \in N^c.
\end{align*}
Therefore,
\begin{align}\label{maxtLog3}
\sup_{t> \log 2}|K_2(t,x,y)|
&\lesssim
e^{|y|^2-|x|^2}\Big[|x|^{1-n}+e^{ -|y_\perp|^2}\Big(\frac{|y|}{|x|}\Big)^{n-1} |x|\,\chi_{\{|y|\le 2|x|\}}
 \Big], \quad (x,y)\in N^c.\,\
\end{align}

Let now $1/2\le r<1$, that is,  $0<t\le\log2$. We decompose  the kernel in \eqref{eq:Kt2} as 
\begin{align*}
|K_2(t,x,y)|
& \lesssim
\frac{(-\log r) r^{n+1}}{(1-r^2)^{(n+1)/2}}
 |x|
\exp\Big( -\frac{1}{2}\frac{|y-r x|^2}{1-r^2} \Big)   e^{|y|^2-|x|^2}\\
&\qquad \times
\Big(
\chi_{\{s\leq 1/3\}}
+
\chi_{\{s\geq 2\}}
+
\chi_{\{1/3< s  \leq 2\}}
\Big)(r_0)\\
& =: \sum_{j=1}^3K_{2j}(t,x,y), \quad x,y\in\Rn.
\end{align*}
Now we proceed as in \cite[p. 12, cases 2.1 and 2.2]{BrSj}.   Since for $r_0\le 1/3$ and $r_0\ge 2$ we have that $|r-r_0|\gtrsim ( 1+|r_0|)$,
we get
\begin{align}\label{K21+K22}
 K_{21}(t,x,y) + K_{22}(t,x,y)
& \lesssim
\frac{|x|}{(1-r^2)^{(n+1)/2}}   \,
\exp\Big( -\frac{1}{2}\frac{|y-r x|^2}{1-r^2} \Big)
e^{|y|^2-|x|^2} \nonumber\\
&   \lesssim
\frac{ (1+|r_0|)|x|  }{(1-r^2)^{(n+1)/2}}   \,
\exp\Big( - c
\frac{(1+|r_0|)^2|x|^2 +|y_\perp|^2}{1-r^2} \Big)
e^{|y|^2-|x|^2} \nonumber \\
&   \lesssim
\Big( (1+|r_0|)|x| \Big)^{1-(n+1)}e^{|y|^2-|x|^2}
\lesssim \frac{e^{|y|^2-|x|^2}}{|x|^n}.
\end{align}

Consider the case $ r_0=\frac{|y|}{|x|}\cos(\theta(x,y))\in (1/3,2)$.
{By performing the change of variables ${r=\frac{1-s}{1+s}}$, we have that $0<s\le 1/3$ and
\begin{align}\label{K23r}
K_{23}(t,x,y)
&\lesssim
\log\Big(\frac{1+s}{1-s}\Big)
\Big(\frac{1-s}{s^{1/2}}\Big)^{n+1}|x|\;
\exp\Big(-\frac{|(1+s)y-(1-s)x|^2}{8s}\Big)
e^{|y|^2-|x|^2}\nonumber\\
&\lesssim  \frac{(1-s)^{n+1}|x|}{s^{(n-1)/2}}
\exp\Big(-\frac{|(1+s)y-(1-s)x|^2}{8s}\Big)
e^{|y|^2-|x|^2}. 
\end{align}
On the other hand, by using
\begin{align*}
|y(1+s)-x(1-s)|^2
&\gtrsim
\frac{1}{(1+|x|)^2}, \quad 0<s\le \frac{1}{8(1+|x|)^2},\quad  (x,y) \in N^c, 
\end{align*}
we get
\begin{align*}
\sup_{0<s<\frac{1}{8(1+|x|)^2}}K_{23}(t,x,y)|_{t=\log\left(\frac{1+s}{1-s}\right)}
&    \lesssim \sup_{0<s<\frac{1}{8(1+|x|)^2}}
\frac{|x|}{s^{(n-1)/2}}
\exp\Big(-\frac{c}{s(1+|x|)^2}\Big)
e^{|y|^2-|x|^2}\\
&\lesssim
|x|(1+|x|)^{n-1}e^{|y|^2-|x|^2}
\lesssim
(1+|x|)^n e^{|y|^2-|x|^2}, 
\, 
(x,y)\in N^c.
\end{align*}
Moreover,
\begin{align*}
\sup_{\frac{1}{8(1+|x|)^2}<s<1}K_{23}(t,x,y)|_{t=\log\left(\frac{1+s}{1-s}\right)}
&\lesssim
\sup_{\frac{1}{8(1+|x|)^2}<s<1}
\frac{|x|}{s^{(n-1)/2}}
e^{|y|^2-|x|^2} \\
& \lesssim
(1+|x|)^ne^{|y|^2-|x|^2}, 
\quad (x,y)\in N^c.
\end{align*}
}
Therefore, we have shown that
\begin{align}\label{K23tlog2}
\sup_{0<t<\log 2}K_{23}(t,x,y)
\lesssim
 (1+|x|)^ne^{|y|^2-|x|^2}, 
 \quad (x,y)\in N^c.
\end{align}

On the other hand, also for $r_0\in (1/3,2),$ we can proceed as in \cite[p. 12]{BrSj} and consider the following scenarios
\begin{itemize}
    \item[$i)$] $1-r \leq \frac{1}{2}(1-r_0) \vee \frac{3}{2}(r_0-1)$,\\
    \item[$ii)$]  $1-r > \frac{3}{2}(1-r_0)$,\\
    \item[$iii)$] $r_0 <1$ and $|r-r_0| < \frac{1}{2}(1-r_0)$.
\end{itemize}
In all these cases, the following properties will be crucial:
\begin{align}
    &|y-rx|^2=|y_{\perp}|^2+|r-r_0|^2|x|^2, \quad  x,y\in\Rn, \: 0<r<1,\label{identyrx}\\
    &1-r_0=\frac{|x-y_x|}{|x|},\quad x,y\in\Rn,\label{1-r0}\\
    &|y_{\perp}|\ge |x|\sin\theta, \quad \text{whenever } 1/3<r_0<2.\label{yperpsin}
\end{align}

We start analyzing the situation $i)$. Observe that under the current assumptions
$$|1-r_0| \approx |r-r_0|
\quad \text{and} \quad
1-r \lesssim |1-r_0|.$$
Then, by using \eqref{identyrx} we get
\begin{align*}
|y-rx|^2
& = |r-r_0|^2|x|^2+|y_\perp|^2
\gtrsim  (1-r)^2|x|^2+|y_\perp|^2.
\end{align*}
Therefore, since  
$0\le \frac{-\log r}{1-r} \lesssim $ for $1/2\le r <1$, by using \eqref{yperpsin} we deduce
\begin{align}\label{eq:K23i}
K_{23}(t,x,y)
& \lesssim
\frac{(1-r)^{1/2}}{(1-r)^{n/2}}
\, |x| \,
\exp\Big( -c\frac{{(1-r)^2}|x|^2+|y_\perp|^2}{(1-r)(1+r)}\Big) \,
e^{|y|^2-|x|^2} \nonumber \\
& \lesssim
\frac{1}{(1-r)^{n/2}}
e^{ -c(1-r)|x|^2}
e^{-c|y_\perp|^2/(1-r) }\, e^{|y|^2-|x|^2}\nonumber \\
& \lesssim
{\frac{ e^{|y|^2-|x|^2}}{|y_{\perp}|^n}}
\lesssim
\frac{ e^{|y|^2-|x|^2}}{(|x| \sin\theta)^n}, 
\quad(x,y)\in N^c.
\end{align}

Next, we turn to the case $ii)$. Now $|r-r_0| \approx 1-r$, so the estimate \eqref{eq:K23i} above remains valid.\\

Finally, let's treat $iii)$. Now,
$1-r \approx 1-r_0$ and
since \eqref{identyrx} and \eqref{1-r0} hold,   we get
\begin{align}\label{eq:K23iii}
  K_{23}(t,x,y)
  & \lesssim
\frac{|x|}{(1-r)^{(n-1)/2}}
\exp\Big( -c\frac{|r-r_0|^2|x|^2+|y_\perp|^2}{1-r} \Big)  \, e^{|y|^2-|x|^2} \nonumber \\
& \lesssim
\frac{|x|^{(n+1)/2}}{|x-y_x|^{(n-1)/2}}
\exp\Big( -c\frac{|y_\perp|^2|x|}{|x-y_x|} \Big)
  \, e^{|y|^2-|x|^2}, \quad(x,y)\in N^c.
\end{align}
Then, for $|x||x-y_x| < 1$,
\begin{align*}
K_{23}(t,x,y)
& \lesssim  \frac{|x|^{(n+1)/2}}{|x-y_x|^{(n-1)/2}}
\Big(\frac{|x-y_x|}{|y_\perp|^2|x|} \Big)^{n/2}
  \, e^{|y|^2-|x|^2} \nonumber \\
& \lesssim
\frac{(|x||x-y_x|)^{1/2}}{|y_\perp|^n}
  \, e^{|y|^2-|x|^2}\nonumber\\
& \lesssim
\frac{ e^{|y|^2-|x|^2}}{(|x| \sin\theta)^n}, \quad(x,y)\in N^c.
\end{align*}
{Now assume that $|x||x-y_x| \geq 1$. We have to distinguish three cases:
$$
|x|/3\le |y_x|<|x|, \quad
|y_x|\ge |x| 
\quad \text{and} \quad
|y_x|<|x|/3.$$

 Observe that when $|x|/3\le |y_x|<|x|$, estimate \eqref{eq:K23iii} is enough, see \cite[Lemma 4.4]{BrSj}.\\
 
Let us consider the situation $|y_x|\ge |x|.$ Then,  since $|y|^2=|y_x|^2+|y_{\perp}|^2\ge |y_x|^2$
we have that $|y|^2\ge |x|^2$ and from  \eqref{K23r} we get that
\begin{align*}
K_{23}(t,x,y)|_{t=\log\left(\frac{1+s}{1-s}\right)}
& \lesssim
(1-s)^{n+1}
\frac{|x+y|+|x-y|}{{s}^{(n-1)/2}} \\
& \qquad \qquad \times
\exp\Big( -\frac{1}{8}(s|x+y|^2+\frac{1}{s}|x-y|^2)\Big)\,
e^{-(|y|^2-|x|^2)/4}  \, e^{|y|^2-|x|^2}\\
& \lesssim
\frac{(1-s)^{n}}{{s}^{n/2}}
\exp\Big(-\frac{|(1+s)y-(1-s)x|^2}{16s}\Big) \,
e^{|y|^2-|x|^2}, \quad(x,y)\in N^c.
\end{align*}
Then, from \cite[Lemma 3.3.3]{Sa} we get that, for  $|y_x|\ge |x|,$
\begin{align}\label{K23:tlog2:yxgex}
 \sup_{0< t\le \log 2}K_{23}(t,x,y)
 \lesssim
e^{|y|^2-|x|^2}
\Big[
(1+|x|)^n\land(|x|\sin\theta)^{-n}
\Big], \quad(x,y)\in N^c.
\end{align}}

{Suppose now $3|y_x|\le |x|$, $(x,y)\in N^c$.   If $ |y_{\perp}|\le |x|/9$, then since $1/2< r=\frac{1-s}{1+s}<1$, we get
\begin{align*}|(1-s)x|
&\le
|(1+s)y-(1-s)x|
+(1+s)\sqrt{\Big(\frac{|x|}{3}\Big)^2+\Big(\frac{|x|}{9}\Big)^2}
\\&\le
|(1+s)y-(1-s)x|+\frac{2\sqrt{10}}{9}|(1-s)x|,
\end{align*}
that is,
$$\Big(1-\frac{2\sqrt{10}}{9}\Big)|(1-s)x|
\le |(1+s)y-(1-s)x|.$$
Therefore, for {$ |y_{\perp}|\le |x|/9$} and $0<s\le 1/3$
\begin{align*}
K_{23}(t,x,y)|_{t=\log\left(\frac{1+s}{1-s}\right)}
&\lesssim
\frac{(1-s)^n|(1+s)y-(1-s)x|}{s^{(n-1)/2}}
\exp\Big(-\frac{|(1+s)y-(1-s)x|^2}{8s}\Big)
e^{|y|^2-|x|^2} \\
&\lesssim
\frac{(1-s)^n}{s^{n/2}}
\exp\Big(-\frac{|(1+s)y-(1-s)x|^2}{16s}\Big)
e^{|y|^2-|x|^2}, \quad (x,y)\in N^c.
   \end{align*}}
{If $ |y_{\perp}|> |x|/9$, then
\begin{align*}
& K_{23}(t,x,y)|_{t=-\log r=\log\left(\frac{1+s}{1-s}\right)} \\
&  \qquad \lesssim 
(-\log r)
\frac{r^{n+1} |x|}{(1-r^2)^{(n+1)/2}} \exp\Big(-\frac{1}{4}\frac{|r-r_0|^2|x|^2+|y_{\perp}|^2}{1-r^2}\Big)
\exp\Big( -\frac{1}{4}\frac{|y-rx|^2}{1-r^2}\Big)
e^{|y|^2-|x|^2} \\
& \qquad  \lesssim 
 \frac{r^{n}|y_{\perp}|}{(1-r^2)^{(n+1)/2}} 
\exp\Big(-\frac{1}{4}\frac{|r-r_0|^2|x|^2+|y_{\perp}|^2}{1-r^2}\Big)
\exp\Big( -\frac{1}{4}\frac{|y-rx|^2}{1-r^2}\Big)
e^{|y|^2-|x|^2}\\
& \qquad  \lesssim 
 \frac{r^{n} }{(1-r^2)^{n/2}}
\exp\Big( -\frac{1}{4}\frac{|y-rx|^2}{1-r^2}\Big)
e^{|y|^2-|x|^2}\\
&\qquad  \lesssim  \frac{(1-s)^n}{s^{n/2}}
\exp\Big(-\frac{|(1+s)y-(1-s)x|^2}{16s}\Big)
e^{|y|^2-|x|^2}, \quad (x,y)\in N^c.
   \end{align*}
Then, from \cite[Lemma 3.3.3]{Sa} we get that, for  $3|y_x|\le |x|$,
\begin{align}\label{K23:3yx}
 \sup_{0<t\le \log 2}K_{23}(t,x,y)
 &\lesssim 
 e^{|y|^2-|x|^2}
\Big[
(1+|x|)^n\land(|x|\sin\theta)^{-n}
\Big], \quad (x,y)\in N^c.
\end{align}}

\quad \\
By combining \eqref{K23tlog2}, \eqref{eq:K23i},  \eqref{eq:K23iii}, \eqref{K23:tlog2:yxgex} and \eqref{K23:3yx} we conclude that
\begin{align}\label{K23:0.5r1}
& \sup_{0<t\le \log 2}K_{23}(t,x,y)
\lesssim 
e^{|y|^2-|x|^2}
\Big[
(1+|x|)^n\land(|x|\sin\theta)^{-n}\nonumber\\
&\qquad+\chi_{\{|x||x-y_x|\ge 1, \frac{|x|}{3}\le |y_\perp|<|x|\}}(x,y)\frac{|x|^{(n+1)/2}}{|x-y_x|^{(n-1)/2}}
\exp\Big( -c\frac{|y_\perp|^2|x|}{|x-y_x|} \Big)
\Big], \quad (x,y)\in N^c.
\end{align}

\quad \\

From \eqref{maxtLog3}, \eqref{K21+K22} and \eqref{K23:0.5r1}, by using
\cite[Lemmas 4.1-- 4.4]{BrSj} we deduce that $K_{2,glob}^*$ is bounded from $L^1(\Rn, \gamma_{-1})$ into $L^{1,\infty}(\Rn, \gamma_{-1})$.

\quad \\
\textbf{Step 3: $K_{3,glob}^*$  is bounded from $L^1(\Rn,\gamma_{-1})$ into $L^{1,\infty}(\Rn,\gamma_{-1})$.} \\

Since
\begin{equation}\label{eq:te-2t}
\frac{t e^{-2t}}{1-e^{-2t}} \le 1,\quad   t>0,
\end{equation}
we can write
\begin{align*}
|K_3(t,x,y)|
& \lesssim
\frac{e^{-nt}}{(1-e^{-2t})^{(n+2)/2}}
|y-e^{-t}x|^2 \,
 \exp\Big( \frac{-|y-e^{-t}x|^2}{1-e^{-2t}} \Big) \,
 e^{|y|^2-|x|^2}, \quad x,y \in \Rn, \, t>0.
\end{align*}

Next, the change of variables
$t=\log \Big( \frac{1+s}{1-s} \Big),$ $0<s<1,$
allows us to write
\begin{align}\label{K3-3}
|K_3(t,x,y)|
&  \lesssim
\frac{(1-s)^n}{s^{n/2}} \,
 \exp\Big( -\frac{|(1+s)y-(1-s)x|^2}{8s} \Big) \,
 e^{|y|^2-|x|^2} \nonumber\\
&  \lesssim
\Big[ (1+|x|)^n \land (|x| \sin \theta)^{-n}\Big] \,
e^{|y|^2-|x|^2}, \quad (x,y) \in N^c,
\end{align}
where in the last step we used \cite[Lemma 3.3.3]{Sa}. From \eqref{K3-3} and  by using  \cite[Lemma 3.3.4]{Sa} we get that
$K_{3,glob}^*$  is bounded from $L^1(\Rn,\gamma_{-1})$ into $L^{1,\infty}(\Rn,\gamma_{-1})$.

\vspace{0.5cm}
Thus, we have proven that the operator $K^*_{glob}$ defined by
$$K^*_{glob}(g)(x)
:=
\int_{\Rn}   \sup_{t>0}  |t\partial_t T_t^\CA(x,y)| \, \chi_{N^c}(x,y) \, g(y) \, dy ,\quad x\in\Rn,
$$
is bounded from $L^1(\Rn,\gamma_{-1})$ into $L^{1,\infty}(\Rn,\gamma_{-1})$, which clearly implies that the operator
$T_{*,1,glob}^{\CA}$ is also bounded from $L^1_X(\Rn, \gamma_{-1})$ into  $L^{1,\infty}_X(\Rn, \gamma_{-1})$.
\end{proof}

\begin{Lem}\label{Lem:4.9}
For every $1<p<\infty$, there exists $\beta>0$ such that the operator  $T_{*,1,glob}^\CA$ associated with the global region $N_\beta^c$ is bounded from
$L^p_X(\Rn,\gamma_{-1})$ into itself.
\end{Lem}
\begin{proof}
Let $1<p<\infty$.
It is sufficient to see that there exists $\beta>0$ such that the operator $K^{*,\beta}_{j,glob}$ defined as $K^*_{j,glob}$ 
(see \eqref{eq:Kjstar} above) replacing $N$ by $N_\beta$, is bounded from $L^p(\Rn,\gamma_{-1})$ into itself, for $j=1,2,3.$\\

We have that
$$
\sup_{t>0}|K_1(t,x,y)|
\lesssim 
\sup_{t>0}\frac{1}{(1-e^{-2t})^{n/2}}
\exp\Big(-\frac{|x-e^{-t}y|^2}{1-e^{-2t}}\Big), \quad  x,y\in\Rn.
$$

According to \cite[Proposition 2.1]{MPS} we obtain, for every $(x,y)\in N_n^c$,
\begin{align*}
\sup_{t>0}|K_1(t,x,y)|
&\lesssim
\begin{cases}
e^{-|x|^2},  & \mbox{if  $\langle x,y\rangle \le 0$},\\
{|x+y|}^{n}\exp{\left( \frac{|y|^2-|x|^2}{2}-\frac{|x-y||x+y|}{2}\right)}, &\mbox{if $\langle x,y\rangle>0$}.
\end{cases}
\end{align*}
Then, by proceeding as in the proof of Lemma \ref{lem3} we can see that $K_{1,glob}^{*,\beta}$, with $\beta=n$, is bounded from $L^p(\Rn,\gamma_{-1})$ into itself, for every $1<p<\infty.$\\

On the other hand, from \eqref{eq:partialtHt} and \eqref{eq:te-2t} we get that, for every $\eta\in (0,1)$, 
\begin{align*}
|K_3(t,x,y)|
&\lesssim 
\frac{1}{(1-e^{-2t})^{n/2}} 
\exp\Big(-\eta\frac{|x-e^{-t}y|^2}{1-e^{-2t}}\Big)\\
&\lesssim
\frac{1}{(1-e^{-2t})^{n/2}}
\exp\Big(-\eta\frac{|y-e^{-t}x|^2}{1-e^{-2t}}\Big)
e^{\eta(|y|^2-|x|^2)},\quad  x,y\in\Rn, \: t>0.
\end{align*}
Let $\eta\in (0,1)$ and $\beta>0.$ If $(x,y)\in N_\beta^c,$ then
$$
|\sqrt{\eta} x-\sqrt{\eta} y|
\ge
\sqrt{\eta} \beta \Big( 1 \land \frac{\sqrt{\eta}}{\sqrt{\eta}|x|} \Big)
\ge
\beta{\eta}
\Big( 1 \land \frac{1}{\sqrt{\eta}|x|} \Big).
$$
 We choose $\beta>1$ such that $(\sqrt{\eta} x,\sqrt{\eta} y)\in N_n^c$, provided that $(x,y)\in N_\beta^c.$ Then,  by \cite[Proposition 2.1]{MPS} we get, for $(x,y)\in N_\beta^c$,

\begin{align*}
\sup_{t>0}|K_3(t,x,y)|
&\lesssim
e^{\eta(|y|^2-|x|^2)}\begin{cases}
e^{-\eta|y|^2},  & \mbox{if  $\langle x,y\rangle \le 0$},\\
{|x+y|}^{n}\exp{\left(\eta\left( \frac{|x|^2-|y|^2}{2}-\frac{|x-y||x+y|}{2}\right)\right)}, &\mbox{if $\langle x,y\rangle>0$},
\end{cases}\\
&\lesssim
\begin{cases}
e^{-\eta|x|^2},  & \mbox{if  $\langle x,y\rangle \le 0$},\\
{|x+y|}^{n}\exp{\left(\eta\left( \frac{|y|^2-|x|^2}{2}-\frac{|x-y||x+y|}{2}\right)\right)}, &\mbox{if $\langle x,y\rangle>0$}.
\end{cases}
\end{align*}
Let $1<q<\infty.$ We take $1/q<\eta<1$ and $\beta>0$ such that  $(\sqrt{\eta} x,\sqrt{\eta} y)\in N_n^c$, provided that $(x,y)\in N_\beta^c.$ Then, by proceeding as in the proof of Lemma \ref{lem3}, we have that
$$
\sup_{x\in\Rn}\int_{\Rn}
e^{(|x|^2-|y|^2)/q}
\sup_{t>0}|K_3(t,x,y)|\chi_{N^c_\beta}(x,y)dy<\infty.
$$
Moreover,
$$
\sup_{y\in\Rn}\int_{\Rn}
e^{(|x|^2-|y|^2)/q}
\sup_{t>0}|K_3(t,x,y)|\chi_{N^c_\beta}(x,y)dx<\infty.
$$
Then, for $1<p<\infty$ we can find $\beta_p>0$ such that $K^{*,\beta}_{3,glob}$ is bounded from $ L^p(\Rn,\gamma_{-1})$ into itself, for every $\beta>\beta_p.$\\

On the other hand,
$$
|K_2(t,x,y)| 
\lesssim
\frac{te^{-(n+1)t}}{(1-e^{-2t})^{(n+1)/2}}
|x|
\exp\Big(-\eta\frac{|y-e^{-t}x|^2}{1-e^{-2t}}\Big)
e^{|y|^2-|x|^2},\quad  x,y\in\Rn, \: t>0,
$$
for certain $\eta\in (0,1)$. By performing the change of variables $t=\log\frac{1+s}{1-s}$, $t\in (0,\infty)$, it follows that
$$
\frac{|y-e^{-t}x|^2}{1-e^{-2t}}=\frac{|y(1+s)-x(1-s)|^2}{4s}=\frac{1}{4}(s|x+y|^2+\frac{1}{s}|x-y|^2)+\frac{1}{2}(|y|^2-|x|^2), \quad x,y\in\Rn, \: t>0.
$$
Then,
for every   $x,y\in\Rn$ and  $t>0$,
\begin{align*}
& |x|
\exp\Big(-\eta\frac{|y-e^{-t}x|^2}{1-e^{-2t}}\Big) \\
& \qquad \qquad \lesssim 
(|x+y|+|x-y|)
\exp\Big(-\frac{\eta}{4}(s|x+y|^2+\frac{1}{s}|x-y|^2)\Big) 
\exp\Big(-\frac{\eta}{2}(|y|^2-|x|^2)\Big)\\
& \qquad \qquad \lesssim
\frac{1}{\sqrt{s}}\exp{\Big(-\frac{\varepsilon}{4}(s|x+y|^2+\frac{1}{s}|x-y|^2)\Big)}
\exp\Big(-\frac{\eta}{2}(|y|^2-|x|^2)\Big)\\
& \qquad \qquad \lesssim 
\frac{1}{(1-e^{-2t})^{1/2}}
\exp{\Big(-\varepsilon\frac{|y-e^{-t}x|^2}{1-e^{-2t}}\Big)}
\exp{\Big(\frac{\varepsilon-\eta}{2}(|y|^2-|x|^2)\Big)},
\end{align*}
where $\eta\in (0,1)$ and $\varepsilon\in (0,\eta).$
Thus, we get
for  $ x,y\in\Rn$ and $t>0$,
$$
|K_2(t,x,y)| 
\lesssim
\frac{1}{(1-e^{-2t})^{n/2}} 
\exp{\Big(-\varepsilon\frac{|y-e^{-t}x|^2}{1-e^{-2t}}\Big)}
\exp{\Big(\Big(1-\frac{\eta-\varepsilon}{2}\Big)(|y|^2-|x|^2)\Big)},
$$
with $\eta\in (0,1)$ and $\varepsilon\in (0,\eta).$\\

Let $\eta\in (0,1)$ and $\varepsilon\in (0,\eta).$ We choose $\beta>0$ such that
such that $(\sqrt{\eta} x,\sqrt{\eta} y)\in N_n^c$, for  $(x,y)\in N_\beta^c.$ Then,  by \cite[Proposition 2.1]{MPS} we obtain, for $(x,y)\in N_\beta^c$,
\begin{align*}
\sup_{t>0}|K_2(t,x,y)|
&\lesssim
e^{(1-\frac{\eta-\varepsilon}{2})(|y|^2-|x|^2)}
\begin{cases}
e^{-\varepsilon|y|^2},  & \mbox{if  $\langle x,y\rangle \le 0$},\\
{|x+y|}^{n}\exp{\left(\varepsilon\left( \frac{|x|^2-|y|^2}{2}-\frac{|x-y||x+y|}{2}\right)\right)}, &\mbox{if $\langle x,y\rangle>0$},
\end{cases}\\
&\lesssim
\begin{cases}
\exp{\left(\left(1-\frac{\eta+\varepsilon}{2}\right)|y|^2-\left(1-\frac{\eta-\varepsilon}{2}\right)|x|^2\right)},  & \mbox{if  $\langle x,y\rangle \le 0$},\\
{|x+y|}^{n}\exp{\left(\left( 1-\frac{\eta}{2}\right)(|y|^2-|x|^2)-\varepsilon\frac{|x-y||x+y|}{2}\right)}, &\mbox{if $\langle x,y\rangle>0$}.
\end{cases}
\end{align*}
Let $1<q<\infty$. We have that
\begin{align*}
& \int_{\Rn}
e^{(|x|^2-|y|^2)/q}\sup_{t>0}|K_2(t,x,y)|\chi_{N^c_\beta}(x,y)dy \\
& \qquad  \lesssim
\int_{\langle x,y\rangle \le 0}
\exp\Big(-|y|^2\Big(-1+\frac{\eta+\varepsilon}{2}+\frac{1}{q}\Big)\Big)
\exp\Big(-|x|^2\Big(1-\frac{\eta-\varepsilon}{2}-\frac{1}{q}\Big)\Big)
dy\\
& \qquad  \qquad + 
\int_{\langle x,y\rangle >0}{|x+y|}^{n}\exp{\Big(\Big( 1-\frac{\eta}{2}-\frac{1}{q}\Big)(|y|^2-|x|^2)-\varepsilon\frac{|x-y||x+y|}{2}\Big)} \\
& \qquad  \lesssim
 \int_{\Rn}
\exp\Big(-|y|^2\Big(-1+\frac{\eta+\varepsilon}{2}+\frac{1}{q}\Big)\Big)
\exp\Big(-|x|^2\Big(1-\frac{\eta-\varepsilon}{2}-\frac{1}{q}\Big)\Big) dy\\
& \qquad   \qquad+ \int_{\Rn}{|x+y|}^{n}\exp{\Big(-{|x-y||x+y|}\Big(\frac{\varepsilon}{2}- \Big|1-\frac{\eta}{2}-\frac{1}{q}\Big|\Big)\Big)} dy .
\end{align*}
Then, for $0<\varepsilon<\eta<1$ such that $$\frac{\varepsilon+\eta}{2}>1-\frac{1}{q}>\frac{\eta-\varepsilon}{2},$$
we get
$$
\sup_{x\in\Rn}\int_{\Rn}
e^{(|x|^2-|y|^2)/q}
\sup_{t>0}|K_2(t,x,y)|\chi_{N^c_\beta}(x,y)dy<\infty
$$
and
$$
\sup_{y\in\Rn}\int_{\Rn}
e^{(|x|^2-|y|^2)/q}
\sup_{t>0}|K_2(t,x,y)|\chi_{N^c_\beta}(x,y)dx<\infty.
$$
Therefore, for each $1<p<\infty$, there exists $\beta_p>0$ such that the operator $K^{*,\beta}_{2,glob}$ is bounded from $ L^p(\Rn,\gamma_{-1})$ into itself, for every $\beta>\beta_p$.\\

We conclude that for every  $1<p<\infty$ we can find $\beta_p>0$ such that,  for every $\beta>\beta_p$, the operator $T_{*,1,glob}^\CA$ associated with $N_\beta$ is bounded from $ L^p_X(\Rn,\gamma_{-1})$ into itself.
\end{proof}

\begin{proof}[Proof of Proposition \ref{Prop4.5}]
The proof follows the same arguments used in the proof of Proposition \ref{Prop4.1}. We simply observe that
\begin{equation*}
    \sup_{t>0} |t\partial_tW_t(z)|
    \lesssim
    \frac{1}{|z|^n}, \quad z\in\Rn\setminus\{0\}. \qedhere
\end{equation*}
\end{proof}

\subsection{Proof of Theorem \ref{Th:1.1}}\label{subsect:Th1.2}

By using Theorem \ref{Th:1.4} and Propositions \ref{Prop4.1} and \ref{Prop4.5},  we can see that the properties $(a),(b_\alpha)$ and $(c_\alpha)$ are equivalent for $\alpha=0$ and also that $(a)\Rightarrow (b_\alpha)\,\,\,and\,\,\,(c_\alpha)$, for $\alpha=1$.\\

Now, suppose that $(a)$ holds. 
Let 
$1\le p<\infty$
and $f\in L^p_X(\Rn,\gamma_{-1})$. 
We have that
\begin{equation}\label{eq: 4.19}
\int_{\Rn} \frac{\|f(y,\cdot)\|_X}{|x-y|^n}\chi_{N^c}(x,y) dy 
\lesssim 
\frac{1}{(
1 \land 1/|x|)^n
}\|f\|_{L^p_X(\Rn,\gamma_{-1})}, \quad x\in\Rn.
\end{equation}

We denote by 
$\mathbb{Q}_+:=\{r_j\}_{j\in\mathbb{N}}$, the set of rational numbers in $(0,\infty)$. According to \eqref{eq:4.1} and \eqref{eq: 4.19}, for every $k\in\NN$, 
$$
\sup_{j=1,\dots,k}
\Big|W_{r_j}
\Big(f(\cdot,\omega)\chi_{N^c}(x,\cdot)\Big)(x)\Big|\in X, 
\quad x\in\Rn.
$$ 
If $X$ has the Fatou property and
$$
\Big\|\sup_{j=1,\dots,k}
\Big| W_{r_j}
\Big(f(\cdot,\omega)\chi_{N^c}(x,\cdot)\Big)(x)\Big|
\Big\|_X
\lesssim \frac{1}{
(1\land 1/|x|)^n
}\|f\|_{L^p_X(\Rn,\gamma_{-1})}, 
\quad x\in\Rn,
$$
then
$$\sup_{j\in \NN}
\Big|W_{r_j}
\Big(f(\cdot,\omega)\chi_{N^c}(x,\cdot)\Big)(x)\Big|\in X, \quad x\in \Rn.$$
By continuity we deduce that 
$$\sup_{t>0}
\Big|W_{t}
\Big(f(\cdot,\omega)\chi_{N^c}(x,\cdot)\Big)(x)\Big|\in X, \quad x\in\Rn.$$

On the other hand, since $L^p_X(\Rn,\gamma_{-1})\subseteq L^p_X(\Rn,dx)$, Theorem \ref{Th:1.4} implies that $$W_{*,0}(f)(x)\in X,
\quad 
\text{for almost all }x\in\Rn.$$ Therefore, 
$$W_{*,0,loc}(f)(x)\in X,
\quad 
\text{for almost all } x\in\Rn.$$
 This means that
 $$W_{*,0,loc}(f)(x)\in X,
 \quad 
 \text{for every } x\in E \subseteq \Rn,
 \quad
 |\Rn\setminus E|=0.$$
 Moreover, since for every $t>0,$  
$$
|W_{t,0,loc}(f)(x)|
\le 
W_{*,0,loc}(f)(x),
\quad x\in\Rn,$$
we have that 
$$W_{t,0,loc}(f)(x)\in X,
\quad 
\text{for every }
x\in E.$$
Observe that the set $E$ does not depend on $t.$\\

 Now, by using \eqref{TWle1} we have that
 \begin{align*}
|T_{{t},0,loc}^\CA    (f)(x,\omega)- W_{t,0,loc}(f)(x, \omega)|
\lesssim
\int_{\Rn }\frac{1+|x|}{|x-y|^{n-1}}|f(y,\omega)|\chi_{N}(x,y)dy<\infty,
 \end{align*}
for $x\in F$, where 
$F\subseteq \Rn$, $|\Rn\setminus F|=0$ and 
$\omega\in \Omega.$ 
Therefore, for every $t>0,$ 
$$(T_{t,0,loc}^\CA-W_{t,0,loc})(f)(x,\cdot)\in X, \quad x\in F .$$
It follows that 
$$
T_{t,0,loc}^\CA(f)(x,\cdot)\in X,
\quad x\in F\cap E, \quad t>0.$$ Observe that $|\Rn\setminus (E\cap F)|=0.$ \\

Let $\NN:=\{t_k\}_{k=1}^\infty.$ We can write
$$
\sup_{k=1,\dots,\ell}|T_{t_k,0,loc}^\CA(f)(x,\omega)|\in X, \quad x\in F\cap E, \: \omega\in\Omega, \:\ell\in\NN.
$$
Moreover, for $x\in F\cap E$, $\ell\in\NN$ and $\omega\in\Omega,$ we have
\begin{align}\label{Ttlocf}
    \sup_{k=1,\dots,\ell}|T_{t_k,0,loc}^\CA(f)(x,\omega)|&\le |\sup_{k=1,\dots,\ell}|T_{t_k,0,loc}^\CA(f)(x,\omega)|-\sup_{k=1,\dots,\ell}|W_{t_k,0,loc}(f)(x,\omega)||\nonumber\\
    &\qquad+\sup_{t>0}|W_{t,0,loc}(f)(x,\omega)|\nonumber\\
    &\le \sup_{k=1,\dots,\ell}|T_{t_k,0,loc}^\CA(f)(x,\omega)-W_{t_k,0,loc}(f)(x,\omega)|\nonumber\\
    &\qquad+\sup_{t>0}|W_{t,0,loc}(f)(x,\omega)|\nonumber\\
    &\lesssim \int_{\Rn }\frac{1+|x|}{|x-y|^{n-1}}|f(y,\omega)|\chi_{N}(x,y)dy+\sup_{t>0}|W_{t,0,loc}(f)(x,\omega)|.
\end{align}
Then, Fatou's property implies that $$
\sup_{t>0}|T_{t,0,loc}^\CA(f)(x,\omega)|=\sup_{k\in\NN}|T_{t_k,0,loc}^\CA(f)(x,\omega)|\in X, 
\quad x\in F\cap E,$$ 
and the expression in \eqref{Ttlocf} controls 
$$
\sup_{t>0}|T_{t,0,loc}^\CA(f)(x,\omega)|,
\quad x\in F\cap E, \quad \omega\in\Omega.$$

From the arguments above we conclude that 
$$T_{*,0,loc}^\CA(f)(x)\in X,
\quad 
\text{for almost all } x\in\Rn.$$

Since {from \eqref{Ttglob} we have that}
$$
\int_{\Rn}\sup_{t>0}|T_t^\CA(x,y)|\chi_{N^c}(x,y)\|f(y,\cdot)\|_Xdy<\infty, \: \text{ for almost all } {x\in\Rn}, $$
we can deduce that $$T_{*,0,glob}^\CA(f)(x)\in X,
\quad 
\text{for almost all } 
x\in\Rn.$$
Thus, we obtain that
$$
T_{*,0}^\CA(f)(x)\in X,
\quad 
\text{for almost all } x\in\Rn.$$

By proceeding in a similar way we can see that $(a)\Rightarrow (d_\alpha)\,\,\,and\,\,\,(e_\alpha)$, for $\alpha=1$.\\

Moreover, as  in \eqref{3*2} we can see that, for every $\alpha\in (0,1)$, $1\le p <\infty$ and $f\in L^p_X(\Rn,\gamma_{-1}),$
$$
 T_{*,\alpha}^\CA(f)\le  T_{*,1}^\CA(f).
$$
Hence, if $X$ has the Hardy- Littlewood property, we can deduce the properties established in this theorem for $ T_{*,\alpha}^\CA$, $0<\alpha<1$, from the corresponding ones of $T_{*,1}^\CA(f).$\\

Suppose now that $1\le p<\infty$ and that, for every $f\in L^p_X(\Rn,\gamma_{-1})$,  $$T_{*,0}^\CA(f)(x)\in X,
\quad 
\text{for almost all } x\in\Rn.$$ Then, 
$$T_{*,0,loc}^\CA(f)(x)\in X,
\quad 
\text{for almost all } x\in\Rn.$$

 As above we can see that $$W_{*,0,glob}(f)(x)\in X,
 \quad 
 \text{for every } x\in\Rn.$$
 Let $g\in L^p_X(\Rn,dx)$ and  $k\in\NN.$ We denote by $g_k:=g\chi_{B(0,k+1)}.$ It is clear that $g_k\in L^p_X(\Rn,\gamma_{-1})$. Since 
 $$
 T_{*,0,loc}^\CA(g_k)(x)
 =T_{*,0,loc}^\CA(g)(x), 
 \quad |x|\le k,$$
 then
$$ T_{*,0,loc}^\CA(g)(x)\in X,
\quad 
\text{for almost all }x\in\Rn.$$ Moreover, by using the estimates established in the proof of Lemma \ref{lem1}, since for $|x|\le k$, 
 $$
 (T_{*,0,loc}^\CA(g_k)(x)-W_{*,0,loc})(g_k)(x)
 =
 (T_{*,0,loc}^\CA(g)(x)-W_{*,0,loc})(g)(x),$$
 then
 $$
 (T_{*,0,loc}^\CA(g)(x)-W_{*,0,loc})(g)(x)\in X, 
 \quad x\in\Rn.$$ 
 Then, we conclude that
 $$W_{*,0,loc}(g)(x)\in X,
 \quad 
 \text{for almost all } x\in\Rn.$$ Thus, since from \eqref{eq:4.1} we also have that $W_{*,0,glob}({g})(x) \in X$, and we conclude 
 $$W_{*,0}({g})(x)\in X,
 \quad 
 \text{for almost all } x\in\Rn.$$ Theorem \ref{Th:1.4} implies now that $X$ has the Hardy- Littlewood property.
\qed

\section{Proof of Theorem \ref{Th:1.2}}

According to the subordination formula, we have that
$$
P_t^\CA(f)(x):= \frac{t}{{2 }\sqrt{\pi}}
\int_0^\infty
\frac{ e^{- t^2/4u}}{u^{3/2}} T_u^\CA(f)(x)\, du, \quad x\in\Rn,
$$
and
$$
P_t(f)(x):= \frac{t}{{2} \sqrt{\pi}}
\int_0^\infty
\frac{ e^{- t^2/4u}}{u^{3/2}} W_u(f)(x)\, du, \quad x\in\Rn.
$$
Let $\alpha>0$, $1\le p <\infty$ and $f\in L^p_X(\Rn,\gamma_{-1}).$ By using Fubini's theorem we get
\begin{align*}
& \partial_t^\alpha P_{t,loc}^\CA(f)(x)-\partial_t^\alpha P_{t,loc}(f)(x) \\
& \qquad  =
\frac{1}{{2} \sqrt{\pi}}
\int_0^\infty
\partial_t^\alpha(te^{-t^2/4u}) \Big(T_{u,loc}^\CA(f)(x)-W_{u,loc}(f)(x)\Big)\, \frac{du}{u^{3/2}}, \quad  x\in\Rn.
\end{align*}
From \cite[Lemma 3]{BCCFR}, we have that
$$
|\partial_t^\alpha
(t e^{- t^2/4u})|
\lesssim
e^{-t^2/8u}
u^{(1-\alpha)/2}, 
\quad t,u\in (0,\infty).$$

Then, we get that  
$$
|t^\alpha \partial_t^\alpha P_t(z)|
\lesssim
t^\alpha \int_0^\infty
\frac{ e^{-t^2/8u}}{u^{(\alpha+2)/2}}W_u(z)du
\lesssim \frac{1}{|z|^n}, 
\quad z\in\Rn\setminus\{0\}.$$

Also, for every $x\in\Rn$,
\begin{align*}
|\partial_t^\alpha P_{t,loc}^\CA(f)(x)
-\partial_t^\alpha P_{t,loc}(f)(x)|
&\lesssim
\int_0^\infty \frac{ e^{-t^2/8u}}{u^{(\alpha+2)/2}}
|T_{u,loc}^\CA(f)(x)-W_{u,loc}(f)(x)|du\\
&\lesssim
t^{-\alpha}\sup_{u>0}
|T_{u,loc}^\CA(f)(x)-W_{u,loc}(f)(x)|.
\end{align*}
Hence,
\begin{align*}
 \sup_{t>0} \Big|  t^\alpha
 \Big(\partial_t^\alpha P_{t,loc}^\CA(f)(x)-\partial_t^\alpha P_{t,loc}(f)(x)\Big) \Big|
 &\lesssim
 \sup_{u>0}|T_{u,loc}^\CA(f)(x)-W_{u,loc}(f)(x)|.
\end{align*}
    Furthermore, we have that
       \begin{align*}
   \sup_{t>0} |  t^\alpha\partial_t^\alpha P_{t,glob}^\CA(f)(x)|&
   \lesssim
   \sup_{u>0}|T_{u,glob}^\CA(f)(x)|,\; \:x\in\Rn.
    \end{align*}
Now, by using these estimates and some of the results obtained in the proof of Theorems \ref{Th:1.1} and \ref{Th:1.4}, we deduce the following proposition.

\begin{Prop}
Let $\alpha>0.$
\begin{itemize}
    \item[(i)] For $1<p<\infty,$ the operator $P_{*,\alpha}^\CA$ is bounded from $L^p_X(\Rn,\gamma_{-1})$ into itself if, and only if, $P_{*,\alpha}$ is bounded from $L^p_X(\Rn,dx)$ into itself.
    \item[(ii)]The operator $P_{*,\alpha}^\CA$ is bounded from $L^1_X(\Rn,\gamma_{-1})$ into $L^{1,\infty}_X(\Rn,\gamma_{-1})$ if, and only if, $P_{*,\alpha}$ is bounded from ${L^1_X(\Rn,dx)}$ into ${L^{1,\infty}_X(\Rn,dx)}$ .
\end{itemize}
\end{Prop}

\begin{proof}[Proof of Theorem \ref{Th:1.2}.]
From Theorem \ref{Th:1.4}, it follows that
$$
(a)\Leftrightarrow (b_\alpha)\Leftrightarrow (c_\alpha), \,\,\,when\,\,\,\alpha=0
$$
and
$$
(a)\Rightarrow (b_\alpha) \text{ and } (c_\alpha), \text{ for every }\alpha\ge 0.
$$

Also, we have that
\begin{align*}
\sup_{t>0} \Big|  t^\alpha
\Big(\partial_t^\alpha P_t^\CA(x,y)
-\partial_t^\alpha P_{t}(x,y)\Big)\Big|
&\lesssim
\sup_{u>0}|T_{u}^\CA(x,y)-W_{u}(x,y)|,\; \:x,y\in\Rn,
    \end{align*}
and
$$
  \sup_{t>0} |  t^\alpha\partial_t^\alpha P_{t}^\CA(x,y)|\lesssim \sup_{u>0}|T_{u}^\CA(x,y)|,\; \:x,y\in\Rn.
$$
By proceeding as in the proof of Theorem \ref{Th:1.1} we deduce from the estimates above that
\begin{equation*}
(a)\Rightarrow (d_\alpha) \text{ and } (e_\alpha), \text{ for every }\alpha\ge 0.
 \qedhere
\end{equation*} 
\end{proof}

\bibliographystyle{siam}
\bibliography{references}
\quad \\

\end{document}